\documentclass[11pt,reqno]{amsart}
\usepackage[margin=1in]{geometry}
\usepackage{amsmath,amssymb,amsthm}
\usepackage[T1]{fontenc}
\usepackage{lmodern,microtype}
\usepackage{graphicx}
\usepackage{tikz}
\usetikzlibrary{arrows.meta,calc}
\usepackage{xcolor}
\usepackage[hidelinks]{hyperref}
\newtheorem{theorem}{Theorem}
\numberwithin{equation}{section}
\newtheorem{lemma}{Lemma}[section]
\newtheorem{corollary}[theorem]{Corollary}

\newtheorem{remark}{Remark}
\theoremstyle{plain}
\newcommand{\Z}{\mathbb Z}
\newcommand{\Pp}{\mathbb P}
\newcommand{\Ee}{\mathbb E}
\newcommand{\Ecal}{\mathcal E}
\newcommand{\Gcal}{\mathcal G}
\newcommand{\one}{\mathbf 1}
\newcommand{\Pois}{\operatorname{Pois}}
\hypersetup{pdftitle={Percolative MDLA: explosion and linear growth},
 pdfauthor={Eviatar B. Procaccia, Haoyu Wang, Yuan Zhang}}
\title{Percolative MDLA: phase transition and critical behaviour}
\author[Eviatar B. Procaccia]{Eviatar B. Procaccia$^1$}
\address[Eviatar B. Procaccia]{Faculty of Data and Decision Sciences,
Technion - Israel Institute of Technology, Haifa, Israel}
\email{eviatarp@technion.ac.il}
\thanks{$^1$Faculty of Data and Decision Sciences,
Technion - Israel Institute of Technology, Haifa, Israel.}

\author[Haoyu Wang]{Haoyu Wang$^2$}
\address[Haoyu Wang]{Center for Applied Statistics, School of Statistics,
Renmin University of China, Beijing, China}
\email{wanghaoyu123@ruc.edu.cn}

\author[Yuan Zhang]{Yuan Zhang$^2$}
\address[Yuan Zhang]{Center for Applied Statistics, School of Statistics,
Renmin University of China, Beijing, China}
\email{zhang\_probab@ruc.edu.cn}
\thanks{$^2$Center for Applied Statistics, School of Statistics,
Renmin University of China, Beijing, China.}
\date{}
\begin{document}
\begin{abstract}
In this paper we introduce the model of Percolative MDLA (PMDLA) on $\Z^2$ which engulfs the entire occupied percolation component of each
particle that reaches the aggregate. We prove a phase transition of finite time explosion, as a function of the particle density $\lambda$, where the critical value equals
$\lambda_c=-\log(1-p_c^{\rm site}(\Z^2))$.
Below $\lambda_c$, the process is nonexplosive and has a linear upper
bound on its radius. For $2/3<\lambda<\lambda_c$, it also has a linear lower bound on the growth rate. Above $\lambda_c$, an infinite
component is swallowed in finite time a.s. At $\lambda_c$, every update is
finite, but the updates accumulate in finite time, causing a blowup. Excluding criticality, the results are generalized to all higher dimensions. 
\end{abstract}
\maketitle
\setcounter{tocdepth}{1}
\tableofcontents
\section{Introduction}\label{sec:intro}
In this paper we introduce a variant of Multi-particle diffusion-limited aggregation (MDLA), in which a particle reaching the outer boundary causes its entire occupied percolation component to be absorbed. This makes site
percolation relevant to both growth and explosion. We call the process
\emph{Percolative MDLA}, abbreviated PMDLA.
Here explosion includes an accumulation of infinitely many finite
updates. This distinction is essential at criticality.

Classical multi-particle DLA (MDLA) starts from a Poisson cloud of
independent continuous-time random walks. When a particle at a site
$v$ attempts to jump onto the aggregate, $v$ is added and all particles
there become inactive. Kesten and Sidoravicius
\cite[Theorems 2 and 3]{KS} established diffusive growth in one dimension
below density $1$: the front is $O(\sqrt t)$ almost surely, and its
ratio to $\sqrt t$ is bounded away from zero in probability.
Sly \cite[Theorem 1]{sly2020one} proved a positive limiting speed above density
$1$, and obtained a linear lower bound for the diameter in higher
dimensions at the same densities \cite[Corollary 2]{sly2020one}.

At density $1$, Elboim, Nam and Sly \cite[Theorem 1]{ENS} proved
critical growth on the scale $t^{2/3}$ and identified its random
scaling limit. In particular, if $L_t$ denotes the one-dimensional
front, then $t^{-2/3}L_t$ converges in distribution to a positive,
nondegenerate random variable. Their process-level limit is the
integral of a negative power of an $8/3$-dimensional Bessel process.

Sidoravicius and Stauffer \cite[Theorem 1.1]{SS} studied a different
MDLA model, with Bernoulli initial occupation and moving particles
subject to exclusion. In dimensions $d\ge2$ and at sufficiently high
density, they proved with positive probability that the aggregate,
together with the finite regions it encloses, contains a ball growing
linearly in time. Their proof uses competing first-passage percolation
to control the vacant regions encountered by the aggregate.

The Pool model of Cai, Procaccia and Zhang \cite{Pool} converts each
absorbed particle's mass into area while keeping the aggregate a disk.
For its jump-driven version, \cite[Theorem 2]{Pool} proves explosion
above intensity $1$ and nonexplosion at $1$. Both models admit a
conditional Poisson description of the remaining particles.
The geometric growth rules differ: the Pool model prescribes a
circular boundary, whereas PMDLA retains the geometry of the
absorbed occupied components and counts each new site only once.

For Bernoulli site percolation on $\Z^2$, every cluster at $p_c$ is
finite \cite{Nolin,Russo}; see Lemma \ref{lem:percolation}. At $\lambda_c=-\log(1-p_c)$, it ensures that every finite-index PMDLA
engulfment is finite. The critical clusters nevertheless provide
long connections on all scales, as expressed by the crossing and
arm estimates used in Lemma \ref{lem:contacts}.

The contrast at the respective critical parameters is the central
point: critical percolation has no infinite cluster, and the critical
jump-driven Pool model does not explode, yet PMDLA explodes through
an accumulation of finite engulfments. PMDLA combines diffusive
particle supply with instantaneous absorption of percolation
components. Their geometry supplies enough long connections to
make successive macroscopic advances occur in summable time.
Theorem \ref{thm:critical} shows that this interaction can produce
explosion even though each individual critical component is finite. For discussion on the classical DLA and other related aggregation models, see Section \ref{sec:related}.

\subsection{Model and notation}

Vertices of $\Z^2$ are adjacent when their $\ell^1$ distance is one.
Write $e_1=(1,0)$ and $e_2=(0,1)$.
For $F\subset\Z^2$, let
$\partial^oF=\{x\notin F:\|x-y\|_1=1\text{ for some }y\in F\}$.
Let $C(x,O)$ be the nearest-neighbor component of $x$ in $O\subset\mathbb{Z}^2$, with
$C(x,O)=\varnothing$ if $x\notin O$.
We write $d(x,F)=\inf_{y\in F}\|x-y\|_1$,
$d_\infty(x,F)=\inf_{y\in F}\|x-y\|_\infty$, and
$\Lambda_r(v)=\{x\in\Z^2:\|x-v\|_\infty\le r\}$; write $\Lambda_r=\Lambda_r(0)$.
All logarithms are natural except $\log_2$.
We write $a_+=\max\{a,0\}$ and use $\operatorname{Exp}(a)$ for the
exponential law of rate $a$.

Initially each site contains an independent $\Pois(\lambda)$ number
of particles. The particles perform independent rate-one continuous-time
simple random walks $X^i$, and are initially free. Let $O_0$ be the
set of occupied sites. Starting from $\Ecal_{0-}=\{0\}$, set
$\Ecal_0=\{0\}\cup\bigcup_{x\in\partial^o\{0\}}C(x,O_0)$.
If this set is finite, remove all its particles from the free population.

Put $\tau_0=0$. Given the finite aggregate at $\tau_k$, let
$\tau_{k+1}$ be the next time that a free particle enters
$\partial^o\Ecal_{\tau_k}$. At this time, let $z$ be its arrival site
and let $O^-_{\tau_{k+1}}$ be the sites occupied by particles free
just before the update.
Set
$\Ecal_{\tau_{k+1}}=\Ecal_{\tau_k}\cup C(z,O^-_{\tau_{k+1}})$
and remove every particle at a newly swallowed site from the free
population. The aggregate is constant between updates. Removed particles
may continue their walks, but have no further effect.
Note that particles at one swallowed site all leave the free population,
while that site is added only once.

Stop if an infinite component is swallowed. Denote this time by $T_I$,
with $T_I=\infty$ if it never occurs, and let $\tau_k\equiv T_I$ after such an update. The explosion time is
$T=\lim_{k\to\infty}\tau_k$. Thus $T\le T_I$; if $T<T_I$, infinitely
many finite updates accumulate at $T$. if an initial infinite engulfment occurs, then $T=T_I=0$. We use $\inf\varnothing=\infty$.
If an entrance time is infinite, all later update times are also infinite.
No evolution is prescribed after $T$.

Write $\eta_t^f(x)$ for the free particle count after any update at $t$,
$K_t=\Ecal_t\cup\partial^o\Ecal_t$, and
$h_e(t)=\max_{x\in\Ecal_t}x\cdot e$ for
$e\in\{\pm e_1,\pm e_2\}$.
The radius is $R_t=\max_{x\in\Ecal_t}\|x\|_1$, and the buffered
rightmost coordinate is $H_t=\max_{x\in K_t}x_1=h_{e_1}(t)+1$.
Before $T$, the sets $\Ecal_t$ and $K_t$ are finite, connected and
increasing; $K_t$ contains no free particle.
Finally, we use the standard notions of percolation theory. Put $p_c=p_c^{\rm site}(\Z^2)$ to be the critical density of Bernoulli site percolation of $\Z^2$, and also let
$\lambda_c=-\log(1-p_c)$.

\subsection{Main results}

The following theorems describe the explosion phase transition and the growth
bounds for the process in $\Z^2$. 

The first theorem states non explosion in the subcritical regime and gives a linear upper bound on the growth rate. 
\begin{theorem}\label{thm:upper}
If $0<\lambda<\lambda_c$, then $T=\infty$ almost surely. There is a
deterministic $C_\lambda<\infty$ such that
$\limsup_{t\to\infty}R_t/t\le C_\lambda$ almost surely.
\end{theorem}

The second theorem states that in the subcritical regime, given that $\lambda>2/3$, there is a linear lower bound on the growth rate.
\begin{theorem}\label{thm:linear}
If $2/3<\lambda<\lambda_c$, there is a deterministic $v_\lambda>0$ such that
\begin{equation}\label{eq:main}
 \liminf_{t\to\infty}\frac{h_e(t)}t\ge v_\lambda
 \quad\text{almost surely, for every }e\in\{\pm e_1,\pm e_2\}.
\end{equation}
\end{theorem}

\begin{remark}
We expect Theorem \ref{thm:linear} to hold for every
$0<\lambda<\lambda_c$, so that planar PMDLA has linear growth
throughout its subcritical regime. This parallels the conjecture
of linear growth for classical planar MDLA at every positive density
\cite{SS,sly2020one}.
For small $\lambda$, component engulfment usually reduces to one-site attachment,
suggesting similar large-scale growth in the two models. See Figure \ref{fig:sparse} for asparse $\lambda=0.2$ simulation and Figure \ref{fig:slightlysubcritical} for a simulation with slightly subcritical density $\lambda=0.8>2/3$, where Theorem \ref{thm:linear} holds.
\end{remark}

The next theorem gives explosion in the supercritical regime. In particular, it is shown that updates do not accumulate before an infinite component
is swallowed.
\begin{theorem}\label{thm:supercritical}
If $\lambda>\lambda_c$, then $T=T_I<\infty$ almost surely.
\end{theorem}

In the critical regieme, we establish that though every cluster engulfed is finite, there is a.s. infinite accumulation by finite time. 
\begin{theorem}\label{thm:critical}
If $\lambda=\lambda_c$, then $T<\infty$ and $T_I=\infty$ almost surely.
\end{theorem}

\begin{figure}
    \centering
    \includegraphics[width=0.8\linewidth]{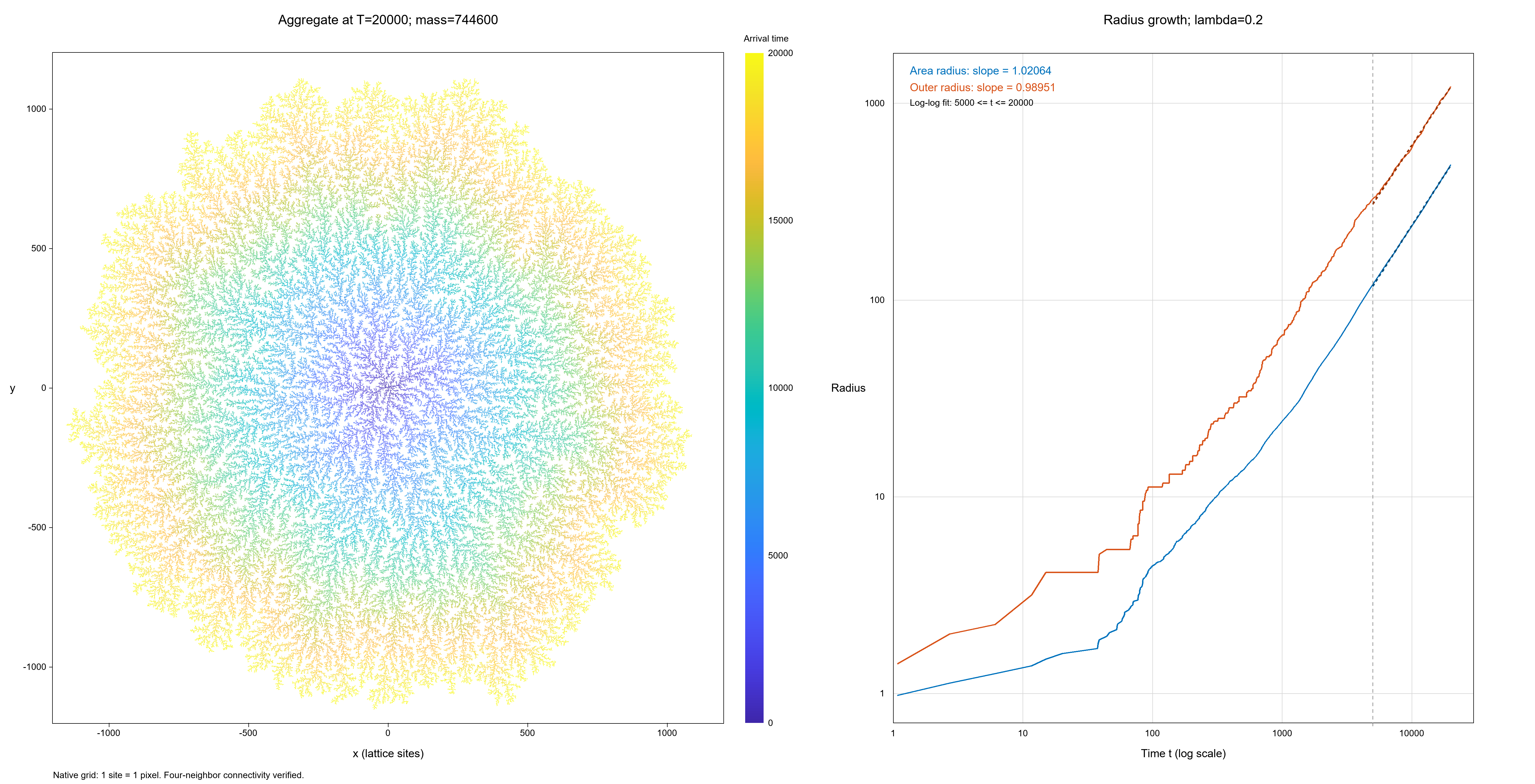}
    \caption{Simulation of PMDLA with $\lambda=0.2$. Panel on the right shows radius vs time on a log-log plot, indicating linear growth}
    \label{fig:sparse}
\end{figure}

\begin{figure}
    \centering
    \includegraphics[width=0.39\linewidth]{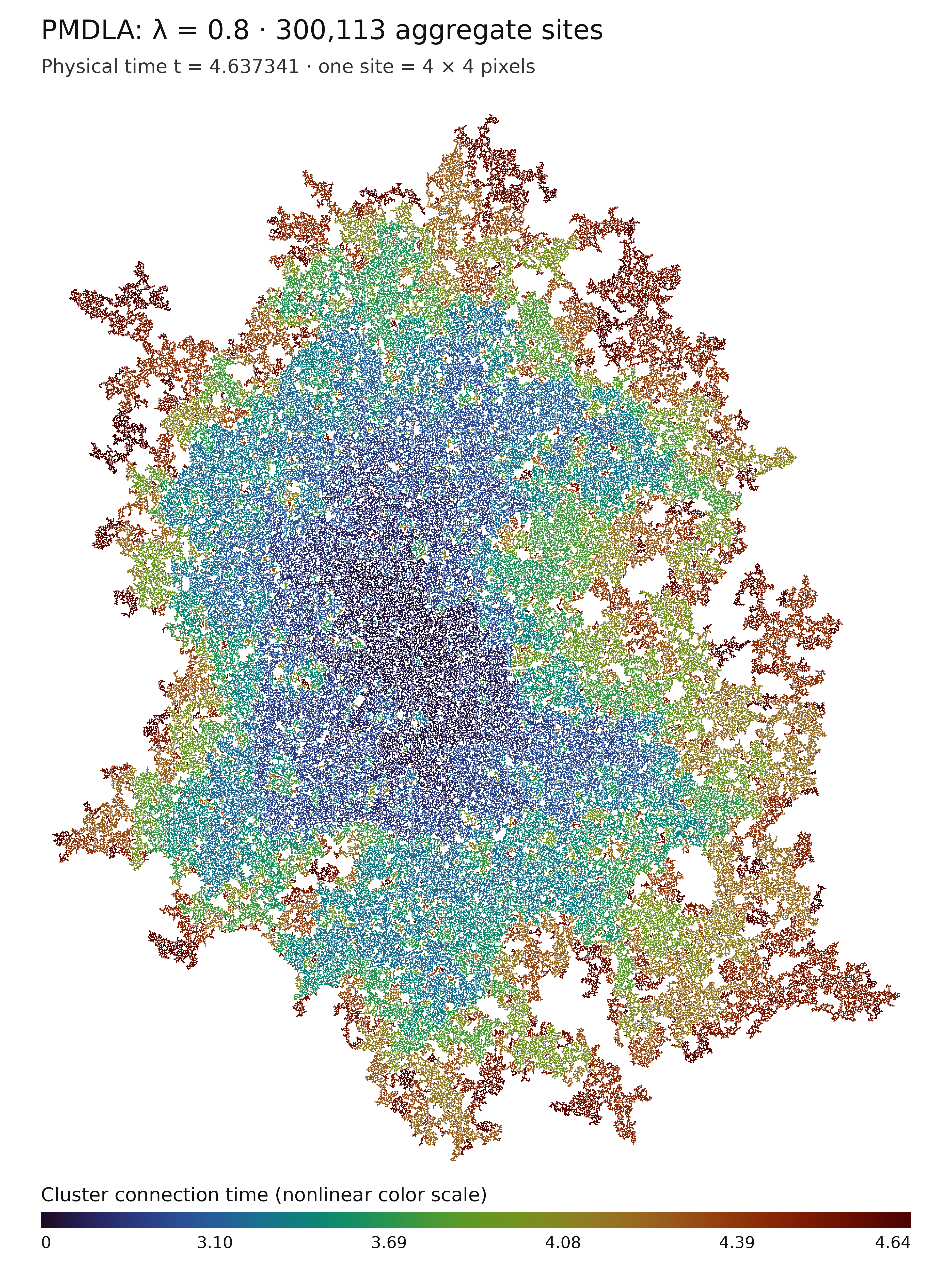}
    \includegraphics[width=0.47\linewidth]{configuration_1px_dark_with_time.png}
    \caption{Simulation of PMDLA with $2/3<\lambda=0.8<\lambda_c$. Aggregate on the left colors connecting clusters according to arrival time. Aggregate on the right contains $10^8$ particles.}
    \label{fig:slightlysubcritical}
\end{figure}
\subsection{Extension to higher dimensions}

For PMDLA on $\Z^d$, write $p_c(d)=p_c^{\rm site}(\Z^d)$ and
$\lambda_c(d)=-\log(1-p_c(d))$.
The quantities $R_t,h_e(t),T,T_I$ retain their definitions, with
$e\in\{\pm e_1,\ldots,\pm e_d\}$.
For an auxiliary simple random walk $X$ of total rate $d$, put
\begin{equation}\label{eq:ad-definition}
 \zeta=\inf\{t\ge0:X_t\cdot e_1=-1\},
 \qquad
 a_d=\frac12\Ee_0\int_0^\zeta\one_{\{X_t=0\}}\,dt.
\end{equation}
This auxiliary clock does not change the rate-one convention for PMDLA.

\begin{corollary}
\label{cor:dimensions}
For rate-one PMDLA on $\Z^d$, $d\ge2$, the following hold.
\begin{enumerate}
\item If $0<\lambda<\lambda_c(d)$, then $T=\infty$ almost surely and
$\limsup_{t\to\infty}R_t/t\le C_{d,\lambda}$ almost surely for a
deterministic $C_{d,\lambda}<\infty$.
\item If $a_d<\lambda<\lambda_c(d)$, there is a deterministic
$v_{d,\lambda}>0$ such that $\liminf_{t\to\infty}h_e(t)/t\ge v_{d,\lambda}$
almost surely for every $e\in\{\pm e_1,\ldots,\pm e_d\}$.
\item If $\lambda>\lambda_c(d)$, then $T=T_I<\infty$ almost surely.
\end{enumerate}
The constants satisfy $a_d<\lambda_c(d)$ for every $d\ge2$ and
$a_2<2/3$.
\end{corollary}

Thus the interval in part (2) is nonempty in every dimension, and
the sufficient planar density is strictly smaller than $2/3$.
\subsection{Open problems}
\begin{enumerate}
    \item First we conjecture a similar behaviour at criticality for all higher dimensions: For all $d>2$, and $\lambda=\lambda_c(d)$, $T<\infty$ and $T_I=\infty$ a.s.
    \item Second, condition the PMDLA to not explode at $\lambda=\lambda_c$ for all $t<t_0$. Take $t_0\to\infty$. What is the growth rate of the conditioned model?
    
    \item Next we restate the famous conjecture for the case of PMDLA: For all $d\ge2$ and $\lambda>0$ there is a deterministic $v_\lambda>0$ such that
\begin{equation}
 \liminf_{t\to\infty}\frac{H_t}t\ge v_\lambda
 \quad\text{almost surely}.
\end{equation}

\end{enumerate}
\subsection{Related work and the role of cluster geometry}\label{sec:related}
There are various aggregation processes that are studies at the heart of probability theory and statistical mechanics. 

Diffusion-limited aggregation (DLA), introduced by Witten and Sander
\cite{WS}, grows by adding one site at a time according to harmonic
measure from infinity. On $\Z^2$, Kesten \cite{Kesten} proved that the
radius of the DLA cluster with $n$ occupied sites is $O(n^{2/3})$ almost surely. Similar results are proven on Stationary DLA, a translation invariant version of DLA \cite{ mu2022scaling, procaccia2020stationary, procaccia2019stationary}. More recently, exact growth bounds, fractal dimension and fluctuations were proven for an off-lattice aggregate grwoing according to a continous harmonic measure, called Stationary Hastings-Levitov \cite{berger2025logarithmic, berger2022growth, chen2025one}. There, it is proved that trees in the aggregate of height $t$, have approximately $t^{3/2}$ particles. 

The Cluster-Cluster model lets entire clusters move at size-dependent
rates and merge when they meet. Berger, Procaccia and Sharon
\cite{BPS} studied its one-dimensional growth, while Berger,
Procaccia, Schmid and Sharon \cite{BPSS} studied its construction
and explosion in higher dimensions. PMDLA also absorbs whole
components, but its free particles move independently: an occupied
component can change before meeting the distinguished aggregate,
rather than moving as a rigid cluster. For the Cluster-Cluster model, the relationship with percolation behaviour is not yet understood in dimensions greater than one. 

This distinction makes cluster geometry especially important. In one dimension, every finite connected cluster is an interval, so its mass essentially determines its spatial extent. In higher dimensions, clusters of the same mass may have very different diameters, boundary sizes, and degrees of ramification. In the present paper, this geometric information is crucial: our proof of blow-up at $\lambda=\lambda_c$ relies on structural properties of critical percolation clusters, rather than on cluster-size estimates alone. Thus, criticality enters not only through the distribution of cluster masses, but also through the clusters' multiscale spatial structure and connectivity. This distinction is evident by the comparison to the Pool model \cite{Pool}, where there is no blowup at criticality (see description in the beginning of Section \ref{sec:intro}).

\subsection{Sketch of the proof ideas}

In this section, we focus on the case of $d=2$ as the key ideas are similar for higher dimensions. 

\emph{Conditional density and the upper bound.}
Exploring only the swallowed components and their vacant boundaries
leaves a conditional Poisson free field. Its intensity is at most
$\lambda$ and equals $\lambda$ times the probability that a reversed
walk avoids the past buffered aggregate. If $\lambda<\lambda_c$,
each engulfment is dominated by a fixed number of subcritical clusters.
Their exponential size tails control the drift of
$\sum_{x\in\Ecal_t}e^{\theta\|x\|_1}$ for small $\theta>0$.
At $\theta=0$ this excludes finite-time accumulation; at $\theta>0$
it gives the linear upper bound.

\emph{Forward growth.}
Following Sly's conforming-particle method \cite{sly2020one}, we use paths
that stay ahead of the aggregate and therefore remain unexamined.
A direct entrance and two short sideways routes advance the rightmost
point. The sideways routes increase the effective small-rate arrival
factor from $1$ to $3/2$. Combining this gain with a random-walk
barrier estimate amplifies a weak growth history whenever
$(3/2)\lambda>1$. Repeating successful advances, with control of the
time needed to recover after failures, gives the linear lower bound.
The same histories imply that the conditional density approaches
$\lambda$ faster than any power of the distance ahead of the aggregate.

\emph{Supercritical explosion.}
From such a history, particles that remain stationary for a short time
form a supercritical occupied field sufficiently far ahead. A finite
occupied connection and one prescribed jump attach its infinite
cluster with a uniform positive probability. Conditional repetition
then gives almost-sure explosion. To exclude an earlier accumulation
of finite updates, complete the stationary field to homogeneous
supercritical percolation by adding finitely many particles. Lemma
\ref{lem:percolation} traps the aggregate in a finite complementary
component until it meets the infinite cluster.

\emph{Critical explosion.}
At $\lambda_c$, each finite-index engulfment is finite.
A planar circuit argument nevertheless supplies a polynomial lower
bound on the sum of probabilities of long occupied connections near
the forward part of a large aggregate. Outside a buffer of width
$O(\log V/\log\log V)$, where $V$ is its forward extension from a
fixed starting line, the conditional field is close enough to critical
percolation to retain this bound. Crossing the buffer costs only
$V^{-o(1)}$. During a fixed trial, advances multiplying $V$ by a fixed
factor have rate at least $cV^\eta$ for some $\eta>0$. The expected
times between geometric levels, stopped at the trial's endpoint, are
summable. A finite forcing event then gives a uniform positive chance
of accumulation within a fixed time; conditional repetition makes
this almost sure.

\subsection{Paper organization} Section \ref{sec:preliminaries} gives the conditional Poisson and percolation
estimates used throughout. Theorem \ref{thm:upper} is proved in
Section \ref{sec:subcritical}. Section \ref{sec:linear} develops the
forward-growth estimates and proves Theorem \ref{thm:linear}.
Sections \ref{sec:supercritical} and \ref{sec:critical} prove
Theorems \ref{thm:supercritical} and \ref{thm:critical}, respectively.
Section \ref{sec:dimensions} proves Corollary \ref{cor:dimensions}.
\section{Preliminaries}\label{sec:preliminaries}

Throughout the proofs, planar walks have total rate $2$, so their first
coordinate has rate $1$. We denote $\Pp_x$ and $\Ee_x$ as the distribution and expectation with respect to the independent walk started at $x$.
Denote by $c,C$ positive constants whose values may change.
Dependencies that require uniformity are stated explicitly.

\subsection{Kurtz density estimates}

Recall that $K_t=\Ecal_t\cup\partial^o\Ecal_t$ and that the free
particle count $\eta_t^f(x)$ is taken after any engulfment at time $t$.
Let $\Gcal_t$ be the completed, right-continuous filtration generated
by the aggregate history, with explosion recorded when it occurs.
It records the update times and the sets added, but need not identify
the particle or edge that caused an update.
We use the rate-two clock throughout this subsection.

The conditional Poisson principle below is attributed to Kurtz in
\cite[Section 1]{KS}; see also \cite[Theorem 1 and Corollary 2.1]{Pool}.
We prove the additional assertion needed for PMDLA: revealing an entire
occupied component and its vacant boundary preserves the Poisson law
outside the enlarged set $K_t$.

\begin{lemma}\label{lem:density}
Suppose the initial particle counts are independent Poisson variables
with mean $\lambda$. On $\{t<T\}$, conditionally on $\Gcal_t$, the
variables $\eta_t^f(x)$, $x\in\mathbb Z^2$, are independent Poisson
with means $u_t(x)\in[0,\lambda]$. These means vanish on $K_t$ and satisfy
\begin{equation}\label{eq:historical}
 u_t(x)=\lambda\Pp_x\bigl(X_r\notin K_{t-r}
                         \text{ for every }0\le r\le t\bigr),
\end{equation}
where $X$ is an independent rate-two simple random walk and the
probability is evaluated with the aggregate history fixed. Moreover,
\begin{equation}\label{eq:deficit}
 \sum_{x\in\mathbb Z^2}(\lambda-u_t(x))
 \le\lambda|K_t|(1+2t)<\infty.
\end{equation}

For deterministic initial means $0\le\mu(x)\le\lambda$, the same
conditional Poisson law holds, with
$u_t(x)=\Ee_x[\mu(X_t);\ X_r\notin K_{t-r}\text{ for every }0\le r\le t]$.
If $D_0=\sum_x(\lambda-\mu(x))<\infty$, the bound in
\eqref{eq:deficit} becomes $\lambda|K_t|(1+2t)+D_0$.
These statements also hold at finite $\Gcal$-stopping times $S$, on
$\{S<T\}$, with the post-update field when $S$ is an update time.
\end{lemma}

We follow the time discretization in \cite[Section~2]{Pool}.
At each step we reveal the particles visiting the current aggregate
or its outer boundary, and add the occupied boundary sites.
Successive steps therefore explore an occupied component one layer
at a time. The marking theorem gives the conditional Poisson law
at every step; we then pass to the original process.

\begin{proof}
We allow initial means $0\le\mu(x)\le\lambda$.
By the marking theorem \cite[Theorem~5.6]{LP}, the prescribed walk
trajectories form a Poisson process $\Pi$ with intensity
$\nu_\mu=\sum_x\mu(x)\Pp_x$.
Write $\mathcal N_t$ for the trajectories of the particles still
free at time $t$.

Fix $t>0$, put $m=2^k$, $h=t/m$, and $s_j=jh$.
On the same trajectories, define
$\Ecal^{(k)}_0=\{0\}$ and $\Pi^{(k)}_0=\Pi$.
Given these objects through step $j-1$, set
$K^{(k)}_{j-1}=\Ecal^{(k)}_{j-1}\cup
\partial^o\Ecal^{(k)}_{j-1}$ and
\begin{align*}
 B^{(k)}_j
 &=\{\omega:\omega([s_{j-1},s_j])\cap K^{(k)}_{j-1}
                 \ne\emptyset\},\\
 \mathfrak I^{(k)}_j
 &=\Pi^{(k)}_{j-1}|_{B^{(k)}_j},
 &\Pi^{(k)}_j&=\Pi^{(k)}_{j-1}-\mathfrak I^{(k)}_j,\\
 \Ecal^{(k)}_j
 &=\Ecal^{(k)}_{j-1}\cup
   \bigcup_{\omega\in\mathfrak I^{(k)}_j}
   \bigl(\omega([s_{j-1},s_j])\cap
                 \partial^o\Ecal^{(k)}_{j-1}\bigr).
\end{align*}
Here $\Pi|_B$ is the restriction to trajectories in $B$, and the
union is over its atoms. Every $\Ecal^{(k)}_j$ is finite and connected.
Particles on its newly exposed boundary are examined at the next step.
In particular, the initial engulfment is included in this construction.

Let $\mathcal H^{(k)}_j=
\sigma(\mathfrak I^{(k)}_1,\ldots,\mathfrak I^{(k)}_j)$ and
$D^{(k)}_j=\bigcap_{\ell=1}^j(B^{(k)}_\ell)^c$.
Conditionally on $\mathcal H^{(k)}_{j-1}$, the set $B^{(k)}_j$
is fixed. The restriction theorem \cite[Theorem~5.2]{LP} makes
the selected and remaining trajectories independent Poisson processes.
Induction gives
\[
 \mathcal L(\Pi^{(k)}_j\mid\mathcal H^{(k)}_j)
 =\operatorname{PPP}(\one_{D^{(k)}_j}\nu_\mu).
\]
The intensity depends only on the sets
$\Ecal^{(k)}_0,\ldots,\Ecal^{(k)}_{j-1}$.
The same law therefore holds conditionally on
$\mathcal G^{(k)}_t=\sigma(\Ecal^{(k)}_j:0\le j\le m)$.
For every bounded measurable $g\ge0$ supported on a set of finite
$\nu_\mu$-measure, this gives
\begin{equation}\label{eq:discrete-poisson}
 \Ee\bigl[e^{-\Pi^{(k)}_m(g)}\mid\mathcal G^{(k)}_t\bigr]
 =L^{(k)}_t(g):=
 \exp\left\{-\int(1-e^{-g(\omega)})
                  \one_{D^{(k)}_m}(\omega)\,\nu_\mu(d\omega)\right\}.
\end{equation}
Here for a point measure $\Pi=\sum_{i\in I}\delta_{\omega_i}$ and a
nonnegative measurable function $g$, write
$\Pi(g):=\int g\,d\Pi=\sum_{i\in I}g(\omega_i)$,
where atoms are counted with multiplicity.
Indeed, if $Z\sim\operatorname{Pois}(a)$ and $c\ge0$, then
$\Ee[e^{-cZ}]
=e^{-a}\sum_{n\ge0}(ae^{-c})^n/n!
=\exp\{-a(1-e^{-c})\}$.
Conditionally on $\mathcal G^{(k)}_t$, the exploration leaves
independent Poisson counts in disjoint measurable sets $B_j$, with
means $\nu_\mu(B_j\cap D^{(k)}_m)$.
The preceding identity therefore proves \eqref{eq:discrete-poisson}
for $g=c\one_B$ with $\nu_\mu(B)<\infty$.
Independence then gives the formula for nonnegative simple functions
supported on sets of finite $\nu_\mu$-measure.

For an arbitrary nonnegative measurable $g$, choose such simple
functions $g_n\uparrow g$, using the $\sigma$-finiteness of $\nu_\mu$.
Then $\Pi^{(k)}_m(g_n)\uparrow\Pi^{(k)}_m(g)$.
Conditional bounded convergence on the left and monotone convergence
in the integral on the right give \eqref{eq:discrete-poisson}.

When taking the limit $h\to 0$, we justify the conditional Poisson field by stopping at finite aggregate sizes.
For a finite set $F$, write $H_F=\inf\{r\ge0:X_r\in F\}$.
Translation invariance and the union bound give
\begin{equation}\label{eq:finite-visitors}
 \nu_\mu(H_F\le t)
 \le\lambda\sum_x\Pp_x(H_F\le t)
 \le\lambda|F|\Ee_0|\{X_r:0\le r\le t\}|
 \le\lambda|F|(1+2t).
\end{equation}
Thus only finitely many prescribed trajectories visit a finite box
by time $t$, and they make finitely many jumps. Their jump times are
distinct almost surely. These facts hold simultaneously for integer
boxes and integer time horizons. They also show directly that, after a finite
engulfment, the vacant boundary has a strictly positive waiting time
before the next entrance.

Let $\sigma_N$ be the first update at which the original aggregate $\Ecal_t$ has at least $N$ vertices, including an initial or infinite engulfment.
Set $\sigma_N=\infty$ if there is no such update.
Until this size is reached, the aggregate and its boundary lie in
the deterministic box $\Lambda_{N+1}=[-N-1,N+1]^2\cap\mathbb Z^2$.
There are only finitely many relevant trajectories and jump times in
this box. Between two such jumps, the discrete construction adds
one occupied boundary layer at each step. Within $N+2$ steps it
either completes the finite engulfment and examines its boundary,
or reaches size $N$.
For all sufficiently large $k$, these steps fit between consecutive
relevant jumps. Consequently, 
\begin{align*}
 \one_{\{|\Ecal^{(k)}_m|<N\}}
 &\longrightarrow\one_{\{\sigma_N>t\}},\\
 \Ecal^{(k)}_{\lfloor q/h\rfloor}&=\Ecal_q
 \quad\text{eventually on }\{\sigma_N>t\},
 \quad\text{for each fixed }0<q\le t,\\
 \Pi^{(k)}_m&=\mathcal N_t
 \quad\text{eventually on }\{\sigma_N>t\}.
\end{align*}
These assertions hold almost surely: no particle jumps at a prescribed
positive time. If an engulfment reaches size $N$, only its first
$N$ vertices are needed for the first assertion, even when the whole
component is infinite.

On $\{\sigma_N>t\}$, the same finite-layer argument shows that,
for $\nu_\mu$-almost every test trajectory,
\[
 \one_{D^{(k)}_m}(\omega)
 \longrightarrow
 \one_{\{\omega(r)\notin K_r\text{ for every }0\le r\le t\}}.
\]
Indeed, a test walk has finitely many jumps on $[0,t]$ and no jump
at any of the finitely many aggregate update times, with the history
fixed. A visit to a newly enlarged $K_r$ is therefore detected once
the grid is sufficiently fine. Dominated convergence yields
$L^{(k)}_t(g)\to L_t(g)$, where
\[
 L_t(g)=\exp\left\{-\int(1-e^{-g(\omega)})
  \one_{\{\omega(r)\notin K_r,\ 0\le r\le t\}}
  \,\nu_\mu(d\omega)\right\}.
\]

To retain the conditioning in the limit, let $\Phi$ be any bounded
function of finitely many $\Ecal_q$, $0<q\le t$, and let $\Phi_k$
be the same function of $\Ecal^{(k)}_q$.
Multiply \eqref{eq:discrete-poisson} by
$\one_{\{|\Ecal^{(k)}_m|<N\}}\Phi_k$ and take expectations.
Bounded convergence gives
\[
 \Ee[\one_{\{\sigma_N>t\}}\Phi e^{-\mathcal N_t(g)}]
 =\Ee[\one_{\{\sigma_N>t\}}\Phi L_t(g)].
\]
Such functions generate the aggregate history: its value at $0$
is recovered by right continuity. A monotone-class argument extends
the identity to all bounded functions of that history.
Finally, $\{\sigma_N>t\}\uparrow\{t<T\}$, because every update
adds a vertex. Letting $N\to\infty$ proves the conditional trajectory
law for the natural aggregate filtration.

For finite $F$, $\nu_\mu(\omega(t)\in F)\le\lambda|F|$.
Mapping trajectories to their positions at time $t$
\cite[Theorem~5.1]{LP} therefore gives independent Poisson counts
with means
\[
 u_t(x)=\sum_a\mu(a)\Pp_a\bigl(X_t=x,\ X_r\notin K_r,
                                      \ 0\le r\le t\bigr)
 =\Ee_x\bigl[\mu(X_t);\ X_r\notin K_{t-r},\ 0\le r\le t\bigr].
\]
The second equality is time reversal of the symmetric walk, with
the aggregate history fixed. It proves \eqref{eq:historical},
the inhomogeneous formula, $0\le u_t\le\lambda$, and
$u_t=0$ on $K_t$. The case $t=0$ follows by the same layer
exploration with the particles fixed at their initial sites.

The counts and the means are right-continuous before $T$.
For the means this follows from the displayed formula: between
updates they evolve by the walk kernel killed on the fixed set $K$.
The surviving trajectory measure is constant between updates.
Testing the identities at times $t_n\downarrow t$ against events
in $\bigcap_n\sigma(\Ecal_r:r\le t_n)$ and using bounded convergence
therefore extends them to the completed, right-continuous filtration
$\Gcal_t$.

For a finite $\Gcal$-stopping time $S$, put
$S_n=2^{-n}\lceil2^nS\rceil$.
If $A\in\Gcal_S$ and $f\ge0$ has finite support, summing the
deterministic-time identity over the possible values of $S_n$ gives
\[
 \Ee\left[\one_A\one_{\{S_n<T\}}
       e^{-\sum_x f(x)\eta^f_{S_n}(x)}\right]
 =\Ee\left[\one_A\one_{\{S_n<T\}}
       e^{-\sum_xu_{S_n}(x)(1-e^{-f(x)})}\right].
\]
Since $S_n\downarrow S$, bounded convergence proves the same identity
at $S$ on $\{S<T\}$. Applying this argument to the trajectory
Laplace functional gives the conditional Poisson law on trajectory
space at $S$ as well.
The field is the post-update field by right continuity.

Finally, $K_r\subset K_t$ for $r\le t$, so
\[
 \lambda-u_t(x)
 \le\lambda\Pp_x(H_{K_t}\le t)
       +\Ee_x[\lambda-\mu(X_t)].
\]
Sum over $x$, apply \eqref{eq:finite-visitors} with $F=K_t$,
and use $\sum_x\Ee_x[\lambda-\mu(X_t)]=D_0$.
This proves the deficit bounds, including at stopping times.
\end{proof}

\begin{remark}[Entrance law and holding times]
The preceding proof establishes the Poisson law after each finite
update, without assuming the law at the next entrance.
To obtain that law, condition at an update time $\sigma<T$ and
mark each remaining trajectory by its first subsequent entrance time
into the fixed set $K_\sigma$, together with its directed entrance
edge. The resulting marked process is Poisson. Given the first
entrance time and edge, removing its trajectory leaves a Poisson
process restricted to later entrances
\cite[Theorem~7.4]{LP}; trajectories with no entrance form an
independent restriction. Thus the other particles, immediately
before that jump, have independent Poisson counts with means
$u_{t-}(x)$. Immediately after the jump and before the component
is revealed, there is one additional particle at the arrival site.

Since each directed nearest-neighbor jump has rate $1/2$, the
predictable entrance intensity at $w\in\partial^o\Ecal_{t-}$ is
$r_t(w)=\frac12\sum_{y\sim w}u_{t-}(y)\le2\lambda$.
After a finite update, the boundary is vacant; hence
\[
 \Pp(\tau_{k+1}-\tau_k\le h\mid\Gcal_{\tau_k})
 \le2\lambda|\partial^o\Ecal_{\tau_k}|h.
\]
This implies $\tau_{k+1}>\tau_k$ almost surely, but does not
exclude accumulation of infinitely many updates.
\end{remark}

We also use independent marking and thinning
\cite[Theorems 5.6 and 5.8]{LP}.
Writing $\nu_\mu=\sum_x\mu(x)\Pp_x$, for disjoint deterministic
measurable sets $B_1,\ldots,B_m$
of trajectories with $\nu_\mu(B_j)<\infty$, the numbers of initial
walk trajectories in these sets are independent Poisson variables
with means $\nu_\mu(B_j)$.
The sets of trajectories and any additional conditioning will be
specified whenever this fact is used.

\subsection{Planar percolation estimates}

The matching lattice (or $*$-connected lattice) of $\Z^2$ joins sites at $\ell^\infty$ distance one.
An open path always means a nearest-neighbor path; a closed matching
path is a closed path on the matching lattice. Half-planes below have boundaries
parallel to a coordinate axis. Write $\Pp_p$ for independent site
percolation with open probability $p$, and $\Ee_p$ for its expectation.
More generally, $\Pp_q$ allows site-dependent open probabilities $q(x)$.
In this setting $C(0)$ denotes the open component of the origin,
empty when the origin is closed.
Product measures satisfy Harris's inequality: increasing events are
positively correlated, and an increasing and a decreasing event are
negatively correlated \cite[Section 2.2]{Grimmett}.

The following lemma states facts on site percolation used in this paper.
\begin{lemma}\label{lem:percolation}
The following statements hold on $\Z^2$.
\begin{enumerate}
\item If $p<p_c$, the size of $C(0)$ has an exponential tail.
\item At $p_c$, all open components are almost surely finite. Open circuits and
closed matching circuits in annuli of fixed aspect ratio have
probabilities bounded away from $0$. For some $C,\delta>0$, a closed matching
arm from a fixed-size box to distance $r$ has probability at most
$Cr^{-\delta}$.
\item If $p>p_c$, a prescribed boundary vertex of a half-plane belongs
to an infinite open cluster within that half-plane with probability
$\theta_H(p)>0$. In the full plane there is almost surely a unique
infinite open cluster $\mathcal I$, and every nearest-neighbor component
of $\Z^2\setminus\mathcal I$ is finite. Uniqueness and finiteness of
the complementary components survive every finite modification.
\end{enumerate}
Also $p_c\ge1/2$, and hence $\lambda_c\ge\log2>2/3$.
\end{lemma}

\begin{proof}
Part (1) is \cite[Theorem 3]{AV}.
Part (2) consists of standard consequences of the critical
Russo--Seymour--Welsh estimates; see
\cite[Theorem 2 and Section 8.1]{Nolin} for these estimates and
their extension to square-site percolation and its matching lattice.
Indeed, open circuits in disjoint annuli give the polynomial bound
on closed matching arms, while closed matching circuits at
arbitrarily large scales exclude infinite open components.

For the half-plane $H=\{x\in\Z^2:x_1\ge0\}$, the usual Peierls
argument gives $p_c(H)<1$; see \cite{Grimmett}.
Apply \cite[Theorem A]{GM} with $F=H$.
Since $2kH+([-k,k]^2\cap\Z^2)$ is a translate of $H$, that theorem
gives $p_c(H)\le p_c+\varepsilon$ for every $\varepsilon>0$.
The reverse inequality follows from $H\subset\Z^2$.
Thus $p_c(H)=p_c$.
For $p>p_c$, any prescribed vertex of $H$ has positive probability
of belonging to an infinite open component: a finite open path
connects it to a vertex with positive percolation probability.

Existence of an infinite full-plane cluster follows from $p>p_c$,
and uniqueness is \cite[Theorem 2]{BK}.
The supercritical planar crossing estimates imply that almost surely
there are open circuits surrounding boxes of arbitrarily large
radius; see \cite[Lemma 37 and Section 8.1]{Nolin}.
Fix a vertex $v$ of the infinite cluster $\mathcal I$.
Every such circuit enclosing $v$ meets an infinite open path from
$v$, and therefore belongs to $\mathcal I$.
Consequently, every vertex outside $\mathcal I$ is surrounded by a
circuit in $\mathcal I$. Its nearest-neighbor component in
$\Z^2\setminus\mathcal I$ is confined inside that circuit and is finite.

For $p<1$, prescribing the states of finitely many sites gives a
probability law absolutely continuous with respect to $\Pp_p$.
The preceding almost-sure properties therefore remain valid after
each fixed finite modification. There are only countably many
finite sets and assignments of their states, so the properties
hold simultaneously after every finite modification.
The case $p=1$ is immediate.

Finally, Russo's matching relation gives $p_c+p_c^*=1$
\cite{Russo}, where $p_c^*$ is the matching-lattice threshold.
Since the matching lattice contains the square lattice,
$p_c^*\le p_c$. Hence $p_c\ge1/2$ and
$\lambda_c=-\log(1-p_c)\ge\log2>2/3$.
\end{proof}

\section{Subcritical nonexplosion and a linear upper bound}\label{sec:subcritical}

For $\lambda<\lambda_c$, subcritical percolation controls the size of each engulfment.
An exponential spatial summation then gives a linear upper bound on the radius.

\begin{proof}[Proof of Theorem \ref{thm:upper}]
Put
$p=1-e^{-\lambda}<p_c$.
Write $\Ee_p$ for expectation under Bernoulli site percolation of
parameter $p$, and let $\omega$ denote its set of open sites.
By Lemma \ref{lem:percolation}, there is $\theta_0>0$ such that
$\chi_\theta=\Ee_p\left[\sum_{x\in C(0,\omega)}e^{\theta\|x\|_1}\right]<\infty$
for all $0\le\theta\le\theta_0$.
Indeed, $\|x\|_1\le |C(0,\omega)|-1$ for $x\in C(0,\omega)$,
and $|C(0,\omega)|$ has an exponential tail.
As usual, $C(0,\omega)=\emptyset$ if $0\notin\omega$.
Fix $0\le\theta\le\theta_0$ and set
$V_\theta(F)=\sum_{x\in F}e^{\theta\|x\|_1}$ and
$m_\theta=1+4e^\theta\chi_\theta$. Conceptually we will show that for some constant $b_\theta$>0, $dV_\theta/dt\le b_\theta V_\theta$, yielding an explonential bound.

We first bound the increase of $V_\theta$ at an entrance. 
Before $T$, consider a possible entrance $y\to z$ at time $s>0$,
with $\|y-z\|_1=1$ and $z\in\partial^o\Ecal_{s-}$.
Conditionally on $\Gcal_{s-}$, let $N=(N(x))_{x\in\Z^2}$ have
independent Poisson counts with the predictable means
$u_{s-}(x)\le\lambda$ from Lemma \ref{lem:density}, which is the conditional distribution of free particles at time $s-$.
The entrance rate in configuration $N$ is $N(y)/2$. Thus by the Poisson identity \eqref{eq:poissonidentity},
for every nonnegative function $h$ of the count configuration
\[
\Ee[N(y)h(N)\mid\Gcal_{s-}]
=u_{s-}(y)\Ee[h(N+\delta_y)\mid\Gcal_{s-}].
\]
Here $\delta_y(x)=\mathbf 1_{\{x=y\}}$, so $N+\delta_y$ adds one
particle at $y$.
Thus, given the entrance time and directed edge, the configuration
after the jump and before the component search consists of independent
Poisson counts with means $u_{s-}(x)$, together with one additional
particle at $z$.

The occupied sites of the Poisson counts are independent, with
probabilities $1-e^{-u_{s-}(x)}\le p$. Couple this occupied set
as a subset of $\omega$.
The component formed by adding the particle at $z$ is contained in
$\{z\}\cup\bigcup_{w\sim z}C(w,\omega)$, where $w\sim z$ means
that $w$ and $z$ are nearest neighbors.
Translation invariance and the triangle inequality give
\[
\Ee_p\left[\sum_{x\in C(w,\omega)}e^{\theta\|x\|_1}\right]
\le e^{\theta\|w\|_1}\chi_\theta
\le e^{\theta(\|z\|_1+1)}\chi_\theta.
\]
The conditional mean increase at this entrance is therefore at most
$m_\theta e^{\theta\|z\|_1}$, and the engulfed component is finite
almost surely. The same comparison for the initial engulfment gives
$\Ee V_\theta(\Ecal_0)\le m_\theta$.

For $F=\Ecal_{s-}$, the predictable entrance rate at
$z\in\partial^o F$ is
$r_s(z)=\frac12\sum_{y\sim z}u_{s-}(y)\le2\lambda$.
Also,
\[
\sum_{z\in\partial^o F}e^{\theta\|z\|_1}
\le4e^\theta V_\theta(F).
\]
Hence the conditional mean increase of $V_\theta$ per unit time is at
most $b_\theta V_\theta(F)$, where
$b_\theta=8\lambda e^\theta m_\theta$.

To use this bound before excluding accumulation, let
$\sigma_n=\inf\{t\ge0:|\Ecal_t|\ge n\}$ for integers $n\ge1$.
Each update adds at least one vertex, so at most $n$ updates occur
up to and including this stop.
Before the stop, the aggregate is connected, contains $0$, and has
fewer than $n$ vertices; thus $V_\theta(\Ecal_t)\le ne^{\theta n}$.
At an update time $s$, write
$J_s=V_\theta(\Ecal_s)-V_\theta(\Ecal_{s-})$.
For $M>0$, replace each stopped increment by $J_s\wedge M$.
The resulting sum has at most $n$ terms and is integrable.
The compensation formula and the preceding bounds give
\[
\Ee\left[V_\theta(\Ecal_0)
  +\sum_{0<s\le t\wedge\sigma_n}(J_s\wedge M)\right]
\le m_\theta+b_\theta\int_0^t
  \Ee\left[\mathbf 1_{\{s\le\sigma_n\}}
  V_\theta(\Ecal_{s-})\right]ds
\le m_\theta+b_\theta ne^{\theta n}t.
\]
Here the sum runs over update times. The stopped indicator
$\mathbf 1_{\{s\le\sigma_n\}}$ is predictable and includes the
last update, even if that update adds many vertices.
Letting $M\to\infty$ and using monotone convergence shows that
$f_n(t)=\Ee V_\theta(\Ecal_{t\wedge\sigma_n})$ is finite and satisfies
$f_n(t)\le m_\theta+b_\theta\int_0^t f_n(s)\,ds$.

By Gronwall's inequality,
\begin{equation}\label{eq:gronwall}
\Ee V_\theta(\Ecal_{t\wedge\sigma_n})
\le m_\theta e^{b_\theta t}.
\end{equation}

Taking $\theta=0$, for every finite $t$ we obtain
\[
\Pp(\sigma_n\le t)
\le \frac{\Ee|\Ecal_{t\wedge\sigma_n}|}{n}
\le \frac{m_0e^{b_0t}}{n}.
\]
Infinitely many updates by time $t$ would require
$\sigma_n\le t$ for every $n$.
Letting $n\to\infty$ and then taking a countable union over positive
integer $t$ excludes accumulation at any finite time.
Since every individual engulfment is finite, the process is defined
for all time, and $\sigma_n\to\infty$ almost surely.

Now fix $0<\theta\le\theta_0$.
Monotone convergence in \eqref{eq:gronwall} gives
$\Ee V_\theta(\Ecal_t)\le m_\theta e^{b_\theta t}$.
Choose $c>b_\theta/\theta$.
Since $e^{\theta R_t}\le V_\theta(\Ecal_t)$, for every integer $k\ge1$,
\[
\Pp(R_k\ge ck)
\le m_\theta e^{-(\theta c-b_\theta)k}.
\]
These probabilities are summable, and Borel-Cantelli gives $R_k<ck$ eventually almost surely.
For $k\le t<k+1$, monotonicity gives $R_t/t\le R_{k+1}/k$.
Therefore $\limsup_{t\to\infty}R_t/t\le c$ almost surely.
Returning to the original rate-one clock divides this bound by $2$.
\end{proof}

\section{Forward speed and the linear lower bound}\label{sec:linear}
Recall that
$H_t=\max_{x\in K_t}x_1=h_{e_1}(t)+1$.
We first obtain forward increment estimates before explosion.
For $\lambda>2/3$, these estimates preserve a positive density ahead
of the aggregate. In the subcritical regime, they yield a linear
lower bound on its growth. 

\subsection{Forward increments}
Let $W$ be an independent rate-one symmetric walk on $\Z$, started at $0$.
For $t<T$, define
\begin{equation}\label{eq:history}
 Y_t(r)=
 \begin{cases}H_t-H_{t-r},&0\le r\le t,\\
 \infty,&r>t,\end{cases}
 \qquad
 S_t=\frac{\lambda}{2}
 \Pp\bigl(W_r\le Y_t(r)\text{ for all }r\ge0\bigr).
\end{equation}
Here the probability concerns only $W$, with the growth history fixed.
Inequalities between profiles will always mean pointwise inequalities.

Following \cite[Sections 2 and 5]{sly2020one}, call a particle
\emph{conforming} until the first time $r$ at which its first coordinate
is at most $H_r$. Such a particle stays outside every queried set and cannot have
been swallowed. Lemma \ref{lem:density} and time reversal show
that, at time $t<T$, the conditional conforming-particle mean at each
site with first coordinate $H_t+1$ is $2S_t$.
The reversed horizontal path must satisfy the inequality in
\eqref{eq:history}, while summing over its vertical starting
coordinate preserves the initial intensity $\lambda$.
Consequently, for each site $v$ with $v_1=H_{t-}$, conforming jumps
$v+e_1\to v$ have predictable intensity $S_{t-}$.

\begin{lemma}
\label{lem:increment}
For every $0<\rho<3/2$, there are $C,c,\gamma_0>0$ such that,
for every finite stopping time $s$, on $\{s<T\}$, every $\Delta>0$,
and $0<\gamma\le\gamma_0$,
\begin{equation}\label{eq:increment}
 \Pp\left(T>s+\Delta,\ \inf_{s\le t\le s+\Delta}S_t\ge\gamma,\
 H_{s+\Delta}-H_s\le\rho\gamma\Delta\mid\Gcal_s\right)
 \le Ce^{-c\gamma\Delta}.
\end{equation}
\end{lemma}
\begin{proof}
Fix a finite stopping time $\sigma<T$ and let $U$ be the
lexicographically first rightmost vertex of $K_\sigma$.
Since $H_\sigma=\max_{x\in\Ecal_\sigma}x_1+1$, we have
$U\in\partial^o\Ecal_\sigma$.

Condition on $\Gcal_\sigma$. Consider the particles whose paths
satisfy $X_r\cdot e_1>H_r$ for every $0\le r\le\sigma$.
By Lemma \ref{lem:density}, their complete trajectories form a
Poisson process. Continue these sampled trajectories even after
absorption, and record their first visits after $\sigma$ to the
fixed line $\{x_1=H_\sigma\}$.
The visits at $U,U+e_2,U-e_2$ form independent Poisson processes:
each trajectory has only one first visit, so the three processes
use disjoint sets of trajectories. Their common rate is
\[
r_\sigma(t)=\frac{\lambda}{2}
\Pp\bigl(W_r\le H_\sigma-H_{(t-r)\wedge\sigma}
          \text{ for }0\le r\le t\bigr),\qquad t\ge\sigma.
\]
This is the calculation in \eqref{eq:history} with $H$ held constant
after $\sigma$. Thus $r_\sigma$ is deterministic given
$\Gcal_\sigma$, and $r_\sigma(t)=S_t$ before the next height
increase or explosion.

Define
\[
\theta_\sigma
=T\wedge\inf\{t\in[\sigma,T):H_t>H_\sigma\text{ or }S_t<\gamma\},
\qquad
q_\sigma=\inf\{t\ge\sigma:r_\sigma(t)<\gamma\}.
\]
The preceding equality of rates implies $\theta_\sigma\le q_\sigma$.
This uses the infima in these definitions; it does not require
$r_\sigma(q_\sigma)<\gamma$.
If $S_\sigma<\gamma$, then $\theta_\sigma=\sigma$ and the
waiting-time bound below is immediate. Otherwise, proceed as follows.

For each visit at time $t<q_\sigma$, sample an independent uniform
variable $V\in[0,1]$, independent of all trajectories.
Keep a visit to $U$ if $V\le\gamma/r_\sigma(t)$.
Keep a visit to $U\pm e_2$ if this inequality holds and the particle's
next jump is towards $U$.
The next direction is uniform among the four neighbors, independently
of the next holding time, which has distribution $\operatorname{Exp}(2)$.
Consequently, the kept direct visits and kept side visits are
independent Poisson processes with rates
\[
r_\sigma(t)\frac{\gamma}{r_\sigma(t)}=\gamma,
\qquad
2r_\sigma(t)\frac{\gamma}{r_\sigma(t)}\frac14=\frac{\gamma}{2},
\]
respectively. The holding times attached to the kept side visits
remain independent $\operatorname{Exp}(2)$ variables.
After $q_\sigma$, extend these two processes by independent Poisson
processes of rates $\gamma$ and $\gamma/2$, assigning independent
$\operatorname{Exp}(2)$ holding times to the added side visits.

Let $E_1,E_2$ be the waiting times from $\sigma$ to the first points
of these two processes, and let $E_3$ be the holding time attached
to the first side point. Then
\[
E_1\sim\operatorname{Exp}(\gamma),\qquad
E_2\sim\operatorname{Exp}(\gamma/2),\qquad
E_3\sim\operatorname{Exp}(2)
\]
are independent, conditionally on $\Gcal_\sigma$.
Before $\theta_\sigma$, the recorded first visits agree with actual
particle arrivals. A direct visit to $U$ increases $H$.
After a kept side visit, the particle stays at $U\pm e_2$ until
its next jump. Absorption there already increases $H$; otherwise
its prescribed jump to $U$ does so, since $U$ remains on the outer
boundary until $H$ increases.
The added points after $q_\sigma$ cause no difficulty because
$\theta_\sigma\le q_\sigma$. Hence, on this enlarged probability space,
\begin{equation}\label{eq:waiting}
\theta_\sigma-\sigma\le T_\gamma:=\min\{E_1,E_2+E_3\}.
\end{equation}
The selections above are used only for this comparison; they are
not included in $\Gcal$. We have therefore obtained a conditional
waiting-time bound given the original $\Gcal_\sigma$.

For $0<a<1$, put $M_\gamma(a)=\Ee e^{a\gamma T_\gamma}$.
Since $T_\gamma\le E_1$, we have $M_\gamma(a)\le(1-a)^{-1}$.
Using independence and the exponential distributions gives
\begin{align*}
M_\gamma(a)
&=1+\frac{a}{1-a}
  \left(1-\Ee e^{-(1-a)\gamma(E_2+E_3)}\right)\\
&=1+\frac{a}{1-a}
  \left(1-\frac{2}{(3-2a)(2+(1-a)\gamma)}\right)\\
&=\frac{6+(3-2a)\gamma}{(3-2a)(2+(1-a)\gamma)}
  \longrightarrow \frac{3}{3-2a}.
\end{align*}
Since $\rho<3/2$, choose $a\in(0,1)$ such that
$\rho\log[3/(3-2a)]<a$.
There are then $c,\gamma_0>0$ such that
$\rho\log M_\gamma(a)\le a-c$ for $0<\gamma\le\gamma_0$.

Now put $\zeta=T\wedge\inf\{t\in[s,T):S_t<\gamma\}$ and
$\sigma_0=s$.
If $\sigma_k<\zeta$, define
$\sigma_{k+1}=\zeta\wedge
\inf\{t\in(\sigma_k,T):H_t>H_{\sigma_k}\}$;
otherwise set $\sigma_{k+1}=\sigma_k$.
On $\{\sigma_k<\zeta\}$, the waiting-time bound applies with
$\sigma=\sigma_k$; on the complementary event the increment is zero.
In particular, induction shows that every $\sigma_k$ is finite
almost surely, and
\[
\Ee\!\left[e^{a\gamma(\sigma_{k+1}-\sigma_k)}
           \,\middle|\,\Gcal_{\sigma_k}\right]\le M_\gamma(a).
\]
Iterated conditioning yields
$\Ee[e^{a\gamma(\sigma_n-s)}\mid\Gcal_s]\le M_\gamma(a)^n$.

Take $n=\lfloor\rho\gamma\Delta\rfloor+1$.
On the event in \eqref{eq:increment}, we have $\zeta\ge s+\Delta$
and fewer than $n$ height advances by $s+\Delta$, because each
advance increases $H$ by at least one.
Thus $\sigma_n-s\ge\Delta$, and exponential Markov inequality gives
\[
\Pp(\sigma_n-s\ge\Delta\mid\Gcal_s)
\le e^{-a\gamma\Delta}M_\gamma(a)^n
\le M_\gamma(a)e^{-c\gamma\Delta}
\le\frac{e^{-c\gamma\Delta}}{1-a}.
\]
This proves \eqref{eq:increment}, with $C=(1-a)^{-1}$.
\end{proof}

\subsection{History profiles and amplification}
We adapt the profiles in \cite[Lemmas 3.1 and 3.5]{sly2020one},
shifting both branches so that a profile remains usable after a
period without growth. For $u\ge0$ and $\alpha>0$, set
$g(u)=\sqrt u\log_2(u+2)$ and
\begin{equation}\label{eq:profile}
 z_\alpha(r)=\min\{\alpha u,g(u)\},
 \qquad u=(r-\alpha^{-3/2})_+.
\end{equation}
Each $z_\alpha$ is nondecreasing and $\alpha$-Lipschitz, and the
profiles increase with $\alpha$.
These properties follow from the monotonicity of $g(u)/u$ and
$0\le g'(u)\le g(u)/u$ for $u>0$.
In particular, for $A\ge0$,
\begin{equation}\label{eq:fallback}
 \beta\le\alpha,\quad \beta^{-3/2}\ge\alpha^{-3/2}+A
 \quad\Longrightarrow\quad
 z_\beta(r)\le z_\alpha((r-A)_+).
\end{equation}
Monotonicity of $H$ also gives
$Y_{s+A}(r)\ge Y_s((r-A)_+)$ whenever $s+A<T$.
Thus \eqref{eq:fallback} applies even when no growth occurs.

\begin{lemma}
\label{lem:barrier}
For every $0<\varepsilon<1$, all sufficiently small $\alpha>0$
satisfy, with $\Delta_\alpha=\alpha^{-4/3}$,
\begin{equation}\label{eq:barrier}
 \Pp\bigl(W_r\le z_\alpha((r-\Delta_\alpha)_+)
                    \text{ for all }r\ge0\bigr)
 \ge2(1-\varepsilon)\alpha.
\end{equation}
\end{lemma}
\begin{proof}
Put $D=\alpha^{-3/2}+\Delta_\alpha$, and recall that
$g(u)=\sqrt u\log_2(u+2)$. Then
\[
z_\alpha((r-\Delta_\alpha)_+)
=\min\{\alpha(r-D)_+,g((r-D)_+)\}.
\]
For $A=\{W_r\le\alpha(r-D)_+\text{ for all }r\ge0\}$,
\cite[proof of Lemma 3.1]{sly2020one} gives
$\Pp(A)\ge2(1-\varepsilon/2)\alpha$ for sufficiently small $\alpha$.

If $0\le u\le\alpha^{-2}$, then
$\alpha u\le\sqrt u\le g(u)$.
Thus, with
\(
B=\{\exists r\ge D+\alpha^{-2}:W_r>g(r-D)\},
\)
we have
$A\setminus B\subseteq
\{W_r\le z_\alpha((r-\Delta_\alpha)_+)\text{ for all }r\ge0\}$.
It remains to show that $\Pp(B)=o(\alpha)$.

For $a>0$, the identity
$\Ee e^{aW_t}=e^{t(\cosh a-1)}$ and Doob's inequality
\cite[Theorem 4.4.2]{Durrett} give
$\Pp(\max_{v\le t}W_v\ge b)\le e^{-ab+t(\cosh a-1)}$.
Taking $a=\min\{b/(2t),1\}$ and using
$\cosh a-1\le a^2$ for $0\le a\le1$, we obtain
\begin{equation}\label{eq:walk-maximum}
\Pp\bigl(\max_{v\le t}W_v\ge b\bigr)
\le\exp\!\left[-\frac14\min\left\{\frac{b^2}{t},b\right\}\right],
\qquad t,b>0.
\end{equation}

Set $L=\log_2(1/\alpha)$, $u_k=2^k\alpha^{-2}$ and
$t_k=D+2u_k$, $k\ge0$.
For sufficiently small $\alpha$, we have $D\le\alpha^{-2}$ and
$\sqrt u\ge\log_2(u+2)$ for every $u\ge\alpha^{-2}$.
Consequently, $t_k\le3u_k$ and
\[
\min\left\{\frac{g(u_k)^2}{t_k},g(u_k)\right\}
\ge\frac13\log_2^2(u_k+2)
\ge\frac13(k+2L)^2.
\]
Since $g$ is increasing, a crossing on $[D+u_k,D+2u_k]$ implies
$\max_{v\le t_k}W_v>g(u_k)$.
These intervals cover $[D+\alpha^{-2},\infty)$, so
\begin{align*}
\Pp(B)
\le\sum_{k=0}^{\infty}
  \Pp\bigl(\max_{v\le t_k}W_v\ge g(u_k)\bigr)
\le\sum_{k=0}^{\infty}e^{-(k+2L)^2/12}
\le e^{-L^2/3}\sum_{k=0}^{\infty}e^{-k^2/12}
=o(\alpha).
\end{align*}
For sufficiently small $\alpha$, this is at most $\varepsilon\alpha$.
Therefore
\[
\Pp\bigl(W_r\le z_\alpha((r-\Delta_\alpha)_+)
          \text{ for all }r\ge0\bigr)
\ge\Pp(A)-\Pp(B)
\ge2(1-\varepsilon)\alpha,
\]
which proves \eqref{eq:barrier}.
\end{proof}


For the rest of this section, assume $\lambda>2/3$ and choose
$0<\rho<3/2$, $0<\varepsilon<1$, and $\kappa>1$ such that
\begin{equation}\label{eq:lambda}
 1<\kappa<\rho\lambda(1-\varepsilon).
\end{equation}
Set $A_\alpha=\lceil\alpha^{-5/3}\rceil\Delta_\alpha$;
then $\alpha^{-3}\le A_\alpha\le2\alpha^{-3}$ for $\alpha\le1$.

\begin{lemma}\label{lem:boost}
Fix the parameters as in \eqref{eq:lambda}, and write
$\Delta_\alpha=\alpha^{-4/3}$,
$N_\alpha=\lceil\alpha^{-5/3}\rceil$ and
$A_\alpha=N_\alpha\Delta_\alpha$.
There are $C,c,\bar\alpha>0$ such that the following holds for
$0<\alpha\le\bar\alpha$ and every finite stopping time $s$.
Set $t_j=s+j\Delta_\alpha$ and define
\begin{align*}
G(s,\alpha)={}&\{T\le t_{N_\alpha}\}
\cup\left\{T>t_{N_\alpha},\ Y_s\ge z_\alpha,\
\min_{0\le j<N_\alpha}(H_{t_{j+1}}-H_{t_j})
\ge\kappa\alpha\Delta_\alpha\right\}.
\end{align*}
Then $G(s,\alpha)\in\Gcal_{s+A_\alpha}$ and, on
$\{s<T,\ Y_s\ge z_\alpha\}$,
\begin{equation}\label{eq:boost}
\Pp(G(s,\alpha)^c\mid\Gcal_s)
\le\epsilon(\alpha):=C\alpha^{-5/3}e^{-c\alpha^{-1/3}}.
\end{equation}
On $G(s,\alpha)\cap\{T>s+A_\alpha\}$,
$Y_{s+A_\alpha}\ge z_{\kappa\alpha}$ and
$H_{s+A_\alpha}-H_s\ge\kappa\alpha A_\alpha$.
\end{lemma}

\begin{proof}
Write $\Delta=\Delta_\alpha$, $N=N_\alpha$ and $A=A_\alpha$.
Suppose $t_j<T$ and $Y_{t_j}\ge z_\alpha$.
For $t_j\le t\le t_{j+1}$, before $T$, monotonicity gives
$Y_t(r)\ge z_\alpha((r-\Delta)_+)$.
Lemma \ref{lem:barrier} therefore yields
$S_t\ge\gamma:=\lambda(1-\varepsilon)\alpha$ throughout this interval.
Since $\rho\gamma>\kappa\alpha$, Lemma \ref{lem:increment} gives,
for sufficiently small $\alpha$,
\[
\Pp\bigl(T>t_{j+1},\
H_{t_{j+1}}-H_{t_j}<\kappa\alpha\Delta
\mid\Gcal_{t_j}\bigr)
\le Ce^{-c\alpha\Delta}=Ce^{-c\alpha^{-1/3}}.
\]

If the block has height increase at least $\kappa\alpha\Delta$,
then its endpoint still satisfies $Y_{t_{j+1}}\ge z_\alpha$.
Indeed, $z_\alpha(r)=0$ for $r\le\Delta$, while for $r\ge\Delta$,
\(
Y_{t_{j+1}}(r)
\ge\kappa\alpha\Delta+z_\alpha(r-\Delta)
\ge z_\alpha(r),
\)
by the $\alpha$-Lipschitz property of $z_\alpha$ and $\kappa>1$.

For $j<N$, let
\(
F_j=\left\{T>t_j,\ Y_s\ge z_\alpha,\
H_{t_{i+1}}-H_{t_i}\ge\kappa\alpha\Delta
\text{ for every }0\le i<j\right\},
\)
and let $B_j=\{T>t_{j+1},\
H_{t_{j+1}}-H_{t_j}<\kappa\alpha\Delta\}$.
Then $F_j\in\Gcal_{t_j}$ and
$F_j\subseteq\{T>t_j,\ Y_{t_j}\ge z_\alpha\}$.
On $F_0=\{s<T,\ Y_s\ge z_\alpha\}$, failure of $G(s,\alpha)$
requires a first failed block. Thus
\begin{align*}
\Pp(G(s,\alpha)^c\mid\Gcal_s)
\le\sum_{j=0}^{N-1}
\Ee\!\left[\one_{F_j}\Pp(B_j\mid\Gcal_{t_j})
                 \,\middle|\,\Gcal_s\right]
\le NC e^{-c\alpha^{-1/3}}
\le C\alpha^{-5/3}e^{-c\alpha^{-1/3}}.
\end{align*}
This proves \eqref{eq:boost}.

Now assume $G(s,\alpha)$ and $T>s+A$.
Summing the block increments gives
$H_{s+A}-H_s\ge\kappa\alpha A$.
Put $d=\alpha^{-3/2}$ and $d'=(\kappa\alpha)^{-3/2}$.
Since $\alpha^{-3}\le A\le2\alpha^{-3}$ for $\alpha\le1$,
\[
\frac{\log_2 A}{\alpha\sqrt A}
=O\!\left(\alpha^{1/2}\log(1/\alpha)\right)\longrightarrow0.
\]
Taking $\bar\alpha$ small enough, we may assume $A\ge4$ and
\[
\Delta\le d',\qquad d\le A/2,\qquad
\max\{g(2A),\,3\sqrt A\log_2 A\}\le\kappa\alpha A,
\qquad \frac{2\log_2 A}{\sqrt A}\le\alpha.
\]
We check $Y_{s+A}(r)\ge z_{\kappa\alpha}(r)$ for $r\le s+A$;
for larger $r$, the left side is infinite.

For $0\le r\le A$, the interval $[s+A-r,s+A]$ contains
$\lfloor r/\Delta\rfloor$ whole blocks. Hence
\begin{align*}
Y_{s+A}(r)
&\ge\kappa\alpha\Delta\lfloor r/\Delta\rfloor
\ge\kappa\alpha(r-\Delta)_+\\
&\ge\kappa\alpha(r-d')_+
\ge z_{\kappa\alpha}(r).
\end{align*}
For $A\le r\le2A$,
\[
Y_{s+A}(r)\ge H_{s+A}-H_s
\ge\kappa\alpha A\ge g(2A)\ge z_{\kappa\alpha}(r).
\]

Finally, let $r\ge2A$ and put $u=r-A-d$, $v=r-d'$.
Then $u\ge A/2$ and $0<v-u=A+d-d'\le3A/2$.
The function $g(w)/w$ is decreasing, and
\[
0\le g'(w)
=\frac{\log_2(w+2)}{2\sqrt w}
 +\frac{\sqrt w}{(w+2)\log 2}
\le\frac{g(w)}w
\le\frac{2\log_2 A}{\sqrt A},\qquad w\ge A/2.
\]
In particular, $g(u)\le\alpha u$, so $z_\alpha(r-A)=g(u)$.
Moreover,
\[
g(v)-g(u)
=\int_u^v g'(w)\,dw
\le\frac{2(v-u)\log_2 A}{\sqrt A}
\le3\sqrt A\log_2 A
\le\kappa\alpha A.
\]
Consequently,
\[
Y_{s+A}(r)
=H_{s+A}-H_s+Y_s(r-A)
\ge\kappa\alpha A+g(u)
\ge g(v)\ge z_{\kappa\alpha}(r).
\]
This completes the proof.
\end{proof}

\subsection{Good histories and forward density}
\begin{lemma}\label{lem:good}
Suppose $\lambda>2/3$. For every sufficiently small $\alpha_*>0$
and every deterministic $s\ge0$, there is a deterministic
$L(s)\in[0,\infty)$ such that, on $\{s<T\}$,
\begin{equation}\label{eq:good}
\Pp\bigl(T\le s+L(s)\ \text{or}\
\{T>s+L(s),\ Y_{s+L(s)}\ge z_{\alpha_*}\}
\mid\Gcal_s\bigr)\ge\tfrac12.
\end{equation}
\end{lemma}

\begin{proof}
The error bound in Lemma \ref{lem:boost} satisfies
\[
\sum_{i\ge1}\epsilon(\alpha_*\kappa^{-i})
\le C_1\sum_{i\ge1}
 e^{-c_1\alpha_*^{-1/3}\kappa^{i/3}}
\longrightarrow0\qquad\text{as }\alpha_*\downarrow0.
\]
Here the polynomial factor in $\epsilon$ is absorbed into the
exponential, and the series on the right is dominated by a summable
series for $\alpha_*\le1$.
Choose $\alpha_*\le\bar\alpha$ so that this sum is at most $1/2$.

Choose the least $j\ge0$ such that
$(\alpha_*\kappa^{-j})^{-3/2}\ge s$, and put
$\beta_i=\alpha_*\kappa^{i-j}$, $0\le i\le j$.
On $\{s<T\}$, we have $Y_s\ge z_{\beta_0}$: for $r\le s$,
$Y_s(r)\ge0=z_{\beta_0}(r)$, and for $r>s$, $Y_s(r)=\infty$.

Define the deterministic times
$s_0=s$ and $s_{i+1}=s_i+A_{\beta_i}$, $0\le i<j$, and set
$L(s)=s_j-s=\sum_{i=0}^{j-1}A_{\beta_i}$.
If $T>s_i$ and $G(s_\ell,\beta_\ell)$ holds for every $\ell<i$,
then Lemma \ref{lem:boost} gives $Y_{s_i}\ge z_{\beta_i}$.
Conditioning at the first failed amplification therefore gives,
on $\{s<T\}$,
\[
\Pp(T>s_j,\ Y_{s_j}\not\ge z_{\alpha_*}\mid\Gcal_s)
\le\sum_{i=0}^{j-1}\epsilon(\beta_i)
=\sum_{i=1}^{j}\epsilon(\alpha_*\kappa^{-i})
\le\tfrac12.
\]
This is \eqref{eq:good}. If $j=0$, the sums are empty and $L(s)=0$.
\end{proof}
The following lemma provides a high Kurtz density given a good history.
\begin{lemma}
\label{lem:profile-density}
Fix a sufficiently small $\alpha_*>0$, write $z=z_{\alpha_*}$, and
set
$b(d)=\Pp(\sup_{r\ge0}(W_r-z(r))\ge d)$ for integers $d\ge1$.
Then $b$ is nonincreasing, $b(1)<1$, and for every $m>0$ there is
$C_m<\infty$, depending on $\alpha_*$ and $m$, such that
\begin{equation}\label{eq:tail}
b(d)\le C_m d^{-m},\qquad d\ge1.
\end{equation}
For every finite $\Gcal$-stopping time $s$, on
$\{s<T,\ Y_s\ge z\}$,
\begin{equation}\label{eq:profile-density}
u_s(x)\ge\lambda[1-b(x_1-H_s)],\qquad x_1>H_s.
\end{equation}
At update times, $u_s$ and $H_s$ refer to the post-update state.
\end{lemma}

\begin{proof}
Monotonicity of $b$ follows from its definition.
Lemma \ref{lem:barrier}, with $\varepsilon=1/2$, and
$z((r-\Delta_{\alpha_*})_+)\le z(r)$ give
\[
1-b(1)
\ge\Pp\bigl(W_r\le z((r-\Delta_{\alpha_*})_+)
                  \text{ for all }r\ge0\bigr)
\ge\alpha_*>0.
\]

Since $g(u)/u\to0$, there are $d_0\ge1$ and $c_0>0$ such that
\[
z(t)=g(t-\alpha_*^{-3/2})\ge c_0\sqrt t\log t,
\qquad \sqrt t\ge\log t,\qquad t\ge d_0.
\]
Fix an integer $d\ge d_0$ and put $q_k=2^kd$.
If $\sup_{r\ge0}(W_r-z(r))\ge d$, then
$W_r>z(r)+d/2$ for some $r\ge0$.
Splitting the possible times into $[0,d]$ and
$[q_k,2q_k]$, $k\ge0$, and using monotonicity of $z$, we obtain
\[
b(d)\le
\Pp\bigl(\max_{r\le d}W_r>d/2\bigr)
+\sum_{k\ge0}\Pp\bigl(\max_{r\le2q_k}W_r>z(q_k)\bigr).
\]
The preceding lower bound on $z$ gives
\[
\min\left\{\frac{z(q_k)^2}{2q_k},z(q_k)\right\}
\ge c(\log d+k)^2.
\]
Thus \eqref{eq:walk-maximum} yields
\begin{align*}
b(d)
\le e^{-cd}+\sum_{k\ge0}e^{-c(\log d+k)^2}
\le e^{-cd}+e^{-c(\log d)^2}\sum_{k\ge0}e^{-ck^2}
\le C e^{-c(\log d)^2}.
\end{align*}
This is bounded by $C_m d^{-m}$ for every $m>0$.
Increasing $C_m$ to cover $1\le d<d_0$ proves \eqref{eq:tail}.

Finally, fix $x_1>H_s$ and put $d=x_1-H_s$.
For the reversed walk $X$ started at $x$, the process
$W_r=x_1-X_r\cdot e_1$ is a rate-one symmetric walk started at $0$.
On $\{\sup_{r\ge0}(W_r-z(r))<d\}$, for $0\le r\le s$,
\[
X_r\cdot e_1=x_1-W_r
>H_s-z(r)
\ge H_s-Y_s(r)=H_{s-r}.
\]
Hence $X_r\notin K_{s-r}$ throughout $[0,s]$.
The historical density formula in Lemma \ref{lem:density} gives
\[
u_s(x)
=\lambda\Pp_x\bigl(X_r\notin K_{s-r}\text{ for }0\le r\le s\bigr)
\ge\lambda[1-b(d)],
\]
which proves \eqref{eq:profile-density}.
\end{proof}

\subsection{The linear lower bound}
In this section we prove the linear lower bound, assuming $2/3<\lambda<\lambda_c$. The main point is to control the time lost after failed applications
of Lemma \ref{lem:boost}. We record the available history profile by
a level: success improves the profile, while failure leaves a weaker
profile from which the argument continues. The rapidly decreasing
failure probabilities allow us to bound the conditional second
moments of the times between successive returns to a fixed good profile
$z_{\alpha_*}$. At each return, the next attempt succeeds with
conditional probability at least $1/2$, and success increases $H$
by a fixed positive amount. A martingale law of large numbers then
shows that successful attempts occur at a positive rate, giving
the required linear lower bound. 
\begin{proof}[Proof of Theorem \ref{thm:linear}]
Fix $2/3<\lambda<\lambda_c$. By Theorem \ref{thm:upper},
$T=\infty$ almost surely. We prove the lower bound in direction $e_1$.

Choose $\alpha_*>0$ sufficiently small and put
$\alpha_j=\alpha_*\kappa^{-j}$ and $A_j=A_{\alpha_j}$ for $j\ge0$.
Let $k(j)$ be the least integer satisfying
\begin{equation}\label{eq:level}
\alpha_{k(j)}^{-3/2}\ge\alpha_j^{-3/2}+A_j.
\end{equation}
We construct stopping times $T_n$ and
$\Gcal_{T_n}$-measurable levels $J_n$ with
$Y_{T_n}\ge z_{\alpha_{J_n}}$.
Start with $T_0=0$ and $J_0=0$; the profile condition holds because
$Y_0(0)=0$ and $Y_0(r)=\infty$ for $r>0$.
Given $J_n=j$, set $T_{n+1}=T_n+A_j$ and
\[
J_{n+1}=
\begin{cases}
(j-1)_+, & \text{on }G(T_n,\alpha_j),\\
k(j), & \text{otherwise}.
\end{cases}
\]
On success, Lemma \ref{lem:boost} gives the required profile;
at $j=0$, use $z_{\kappa\alpha_*}\ge z_{\alpha_*}$.
On failure, \eqref{eq:fallback} and \eqref{eq:level} give, for $r\ge0$,
\[
Y_{T_{n+1}}(r)
\ge z_{\alpha_j}((r-A_j)_+)
\ge z_{\alpha_{k(j)}}(r).
\]
The last inequality follows from $\alpha_{k(j)}<\alpha_j$ and
the larger initial zero interval in $z_{\alpha_{k(j)}}$.
Since $J_n$ has countably many values and
$G(T_n,\alpha_j)\in\Gcal_{T_n+A_j}$, the construction preserves
the asserted stopping-time and measurability properties.

We next bound the time spent at positive levels.
Set $V(j)=\alpha_j^{-6}$.
Minimality of $k(j)$ and $A_j\le2\alpha_j^{-3}$ give, for
$\alpha_j\le1$,
\[
\alpha_{k(j)}^{-3/2}
<\kappa^{3/2}(\alpha_j^{-3/2}+A_j)
\le3\kappa^{3/2}\alpha_j^{-3},
\qquad
V(k(j))\le C_0V(j)^2,
\]
where $C_0=3^4\kappa^6$.
Thus, on $\{J_n=j\}$ with $j\ge1$,
\begin{equation}\label{eq:drift}
\Ee[V(J_{n+1})\mid\Gcal_{T_n}]
\le\kappa^{-6}V(j)+C_0\epsilon(\alpha_j)V(j)^2
\le qV(j),
\qquad q=\frac{1+\kappa^{-6}}2<1.
\end{equation}
For the last inequality, choose $\alpha_*$ so small that
$C_0\epsilon(\alpha)\alpha^{-6}\le(1-\kappa^{-6})/2$
for $0<\alpha\le\alpha_*$.
Also require $\epsilon(\alpha_*)\le1/2$.

Let $\nu$ be a finite stopping time for the discrete filtration
$(\Gcal_{T_n})_{n\ge0}$, and put
$N=\inf\{m\ge0:J_{\nu+m}=0\}$.
On $\{J_\nu=j\}$ with $j\ge1$, iteration of \eqref{eq:drift} gives
\begin{equation}\label{eq:killed}
\Ee[V(J_{\nu+m})\one_{\{m<N\}}\mid\Gcal_{T_\nu}]
\le q^mV(j),\qquad m\ge0.
\end{equation}
In particular,
$\Pp(N>m\mid\Gcal_{T_\nu})\le q^mV(j)/V(1)$, so $N<\infty$
almost surely. Since $A_i^2\le4V(i)$, conditional triangle
inequality gives
\begin{equation}\label{eq:recovery}
\begin{split}
\left(\Ee\left[
\left(\sum_{m<N}A_{J_{\nu+m}}\right)^2
\middle|\Gcal_{T_\nu}\right]\right)^{1/2}
&\le\sum_{m\ge0}
\left(\Ee[A_{J_{\nu+m}}^2\one_{\{m<N\}}
\mid\Gcal_{T_\nu}]\right)^{1/2}\\
&\le\frac{2\sqrt{V(j)}}{1-\sqrt q}.
\end{split}
\end{equation}
Apply triangle inequality first to finite sums and then use monotone convergence.
The bound is uniform over the history at $T_\nu$.

Define $\nu_0=0$,
$\nu_{n+1}=\inf\{m>\nu_n:J_m=0\}$, and $Q_n=T_{\nu_n}$.
Every interval $[Q_n,Q_{n+1}]$ begins with an attempt of length $A_0$.
If this attempt succeeds, $Q_{n+1}=Q_n+A_0$; otherwise it is
followed by a return from level $k(0)$ to zero.
By \eqref{eq:recovery} and the tower property,
\begin{equation}\label{eq:cycle}
\Ee[(Q_{n+1}-Q_n)^2\mid\Gcal_{Q_n}]\le B,
\qquad
B=2A_0^2+\frac{8V(k(0))}{(1-\sqrt q)^2}.
\end{equation}
Hence all $Q_n$ are finite, and $Q_n\ge nA_0\to\infty$.

Put $\mathcal F_n=\Gcal_{Q_n}$,
$\ell_n=Q_{n+1}-Q_n$, and $I_n=\one_{G(Q_n,\alpha_*)}$.
Both $I_n$ and $\ell_n$ are $\mathcal F_{n+1}$-measurable, and
\[
\Ee[I_n\mid\mathcal F_n]\ge\tfrac12,
\qquad
\Ee[\ell_n\mid\mathcal F_n]\le\sqrt B,
\qquad
H_{Q_n}-H_0\ge d_0\sum_{m<n}I_m,
\]
where $d_0=\kappa\alpha_*A_0$.
For either $Z_n=I_n$ or $Z_n=\ell_n$, the centered variables
$Z_n-\Ee[Z_n\mid\mathcal F_n]$ have uniformly bounded second moments.
Their weighted partial sums
$\sum_{m<n}(Z_m-\Ee[Z_m\mid\mathcal F_m])/(m+1)$
are therefore $L^2$-bounded martingales.
The martingale convergence theorem and Kronecker's lemma
\cite[Theorems 4.4.6 and 2.5.9]{Durrett} give
$n^{-1}\sum_{m<n}(Z_m-\Ee[Z_m\mid\mathcal F_m])\to0$
almost surely. Consequently,
\begin{equation}\label{eq:slln}
\liminf_{n\to\infty}\frac1n\sum_{m<n}I_m\ge\frac12,
\qquad
\limsup_{n\to\infty}\frac{Q_n}{n}\le\sqrt B
\quad\text{almost surely}.
\end{equation}

For $Q_n\le t<Q_{n+1}$, monotonicity gives
$H_t/t\ge H_{Q_n}/Q_{n+1}$. Thus \eqref{eq:slln} yields
\[
\liminf_{t\to\infty}\frac{H_t}{t}
\ge\liminf_{n\to\infty}\frac{H_{Q_n}}{Q_{n+1}}
\ge\frac{d_0}{2\sqrt B}>0
\qquad\text{almost surely}.
\]
The same bound holds for $h_{e_1}(t)=H_t-1$.
By rotation and reflection, it holds simultaneously in the four
coordinate directions. Returning to the original rate-one clock
divides the bound by $2$, so we may take
$v_\lambda=d_0/(4\sqrt B)$.
\end{proof}
\section{Supercritical explosion}\label{sec:supercritical}

Lemma \ref{lem:profile-density} controls the density ahead of histories
satisfying $Y_s\ge z_{\alpha_*}$. We first show that a uniform positive
probability of explosion from such histories implies almost-sure explosion.

\begin{lemma}
\label{lem:repetition}
Let $\lambda>2/3$ and fix $\alpha_*$ as in Lemma \ref{lem:good}.
Suppose there are constants $h>0$ and $p\in(0,1]$ such that, for every
deterministic $s\ge0$,
$\Pp(T\le s+h\mid\Gcal_s)\ge p$
on $\{s<T,\ Y_s\ge z_{\alpha_*}\}$.
Then $T<\infty$ almost surely.
\end{lemma}

\begin{proof}
Fix a deterministic $s\ge0$, put $r=s+L(s)$, and define
$B_r=\{T>r,\ Y_r\ge z_{\alpha_*}\}$.
Since $\{T\le r\}\in\Gcal_r$, the hypothesis gives
\(
\Pp(T\le r+h\mid\Gcal_r)
\ge \mathbf1_{\{T\le r\}}+p\mathbf1_{B_r}.
\)
As $\Gcal_s\subseteq\Gcal_r$, on $\{s<T\}$, the tower property, $p\le1$ and Lemma \ref{lem:good} yield 
\begin{align*}
\Pp(T\le r+h\mid\Gcal_s)
&=\Ee\bigl[\Pp(T\le r+h\mid\Gcal_r)\mid\Gcal_s\bigr]\\
&\ge \Pp(T\le r\mid\Gcal_s)+p\,\Pp(B_r\mid\Gcal_s)\\
&\ge p\,\Pp(\{T\le r\}\cup B_r\mid\Gcal_s)
\ge p/2
.\end{align*}

Define $s_0=0$ and $s_{n+1}=s_n+L(s_n)+h$.
These times are deterministic and $s_n\ge nh$.
Since $\{T>s_n\}\in\Gcal_{s_n}$ and
$\{T>s_{n+1}\}\subseteq\{T>s_n\}$, another application of
the tower property gives
\begin{align*}
\Pp(T>s_{n+1})
=\Ee\bigl[
\mathbf1_{\{T>s_n\}}
\Pp(T>s_{n+1}\mid\Gcal_{s_n})
\bigr]
\le(1-p/2)\Pp(T>s_n).
\end{align*}
Thus $\Pp(T>s_n)\le(1-p/2)^n\to0$, proving
$\Pp(T=\infty)=0$.
\end{proof}

For the supercritical explosion, we first prove $T<\infty$ and then exclude accumulation before an
infinite engulfment. A good history gives particle density close to
$\lambda$ at a constant distance ahead of the aggregate. Particles that
remain stationary over a short interval therefore dominate
supercritical percolation there. A finite occupied connection and
one prescribed jump give a uniform positive probability to link the aggregate to the infinite component, which causes explosion;
Lemma \ref{lem:repetition} then yields almost-sure explosion.
To identify its cause, the finite deficit estimate allows us to
complete a suitable stationary field to homogeneous supercritical
percolation by adding only finitely many particles. The original
stationary field consequently has an infinite cluster with finite
complementary components. Until contact with this cluster, the
aggregate remains in one such finite component, where infinitely
many updates are impossible. Contact causes an infinite engulfment,
proving $T=T_I$. See Figure \ref{fig:supercritical-proof} for visualization of the notations and proof idea. 

\begin{figure}[htbp]
\centering
\begingroup
\definecolor{PMblue}{HTML}{315E83}
\definecolor{PMgreen}{HTML}{24784C}
\definecolor{PMorange}{HTML}{C65B16}
\definecolor{PMgrey}{HTML}{69757D}
\begin{tikzpicture}[
  x=1cm,y=1cm,font=\small,
  line cap=round,line join=round,
  >={Stealth[length=2mm,width=1.4mm]},
  stationary/.style={draw=PMgreen,line width=1.1pt},
  note/.style={align=left,anchor=north west,inner sep=0pt},
  lattice/.style={draw=black!9,line width=.2pt}
]
\path[use as bounding box] (-.1,-.15) rectangle (15.8,11.05);

\node[anchor=west,font=\small\bfseries] at (0,10.8)
  {(a) Connecting to an infinite stationary component};
\fill[PMgreen!6] (5.6,7.15) rectangle (10.25,10.25);
\draw[PMgreen!60,densely dashed] (5.6,7.15)--(5.6,10.25);
\node[anchor=north,font=\footnotesize,text=PMgreen!80!black]
  at (7.9,10.23) {$\{x:x_1\ge H_s+M\}$};

\begin{scope}[shift={(2.6,8.65)},x=.6cm,y=.6cm]
  \foreach \x/\y in {-4/0,-3/-1,-3/1,-2/-2,-2/2,
                       -1/-2,-1/2,0/-1,0/1,1/0} {
    \fill[PMblue!8] (\x-.48,\y-.48) rectangle (\x+.48,\y+.48);
    \draw[PMblue!45,line width=.4pt]
      (\x,\y) circle[radius=.08];
  }
  \foreach \x/\y in {-3/0,-2/-1,-2/0,-2/1,-1/-1,-1/0,-1/1,0/0}
    \path[fill=PMblue!80,draw=white,line width=.5pt]
      (\x-.48,\y-.48) rectangle (\x+.48,\y+.48);
\end{scope}
\node[text=white,font=\normalsize] at (1.65,8.65) {$\mathcal E_s$};
\node[anchor=west,font=\footnotesize,align=left,text=PMblue!90!black]
  at (.25,10.4) {$K_s=\mathcal E_s\cup\partial^o\mathcal E_s$};
\draw[PMblue!70,line width=.5pt] (1.25,10.23)--(1.4,9.93);

\draw[PMgrey,densely dashed] (3.2,7.15)--(3.2,10.05);
\draw[stationary] (3.8,8.65)--(5.6,8.65);
\draw[PMblue,fill=white,line width=.9pt] (3.2,8.65) circle[radius=.075];
\foreach \x in {3.8,4.4,5.0,5.6}
  \fill[PMgreen] (\x,8.65) circle[radius=.07];
\node[below=4pt] at (3.2,8.65) {$U$};
\node[below=4pt] at (3.8,8.65) {$x_1$};
\node[below=4pt] at (4.4,8.65) {$x_2$};
\node[below=4pt] at (5.0,8.65) {$\cdots$};
\node[below=4pt] at (5.6,8.65) {$x_M$};
\node[font=\footnotesize] at (4.4,7.65) {$x_j=U+je_1$};
\node[below,font=\footnotesize] at (3.2,7.15) {$x_1=H_s$};
\node[below,font=\footnotesize] at (5.6,7.15) {$x_1=H_s+M$};

\draw[PMorange,fill=white,line width=1pt]
  (3.8,8.89) circle[radius=.065];
\draw[PMorange,->,line width=1pt]
  (3.76,8.95) to[out=140,in=80] (3.2,8.74);
\node[font=\footnotesize,align=center,text=PMorange!90!black]
  at (4.25,9.62) {first jump $x_1\to U$\\by time $s+h$};

\draw[stationary] (5.6,8.65)--(6.2,8.65)--(6.2,9.25)--(7.4,9.25)
  --(7.4,8.65)--(8.6,8.65)--(8.6,9.25)--(9.2,9.25);
\draw[stationary,->] (9.2,9.25)--(10.0,9.25);
\draw[stationary] (6.8,9.25)--(6.8,9.85);
\draw[stationary] (7.4,8.65)--(7.4,8.05)--(6.2,8.05)--(6.2,7.45);
\draw[stationary] (8.0,8.65)--(8.0,8.05)--(9.2,8.05);
\draw[stationary,->] (9.2,8.05)--(10.0,8.05);
\foreach \x/\y in {6.2/8.65,6.2/9.25,6.8/9.25,7.4/9.25,
  7.4/8.65,8.0/8.65,8.6/8.65,8.6/9.25,9.2/9.25,
  6.8/9.85,7.4/8.05,6.8/8.05,6.2/8.05,6.2/7.45,
  8.0/8.05,8.6/8.05,9.2/8.05}
  \fill[PMgreen] (\x,\y) circle[radius=.055];
\node[text=PMgreen,anchor=west] at (10.02,9.25) {$\infty$};
\node[text=PMgreen,anchor=west] at (10.02,8.05) {$\infty$};

\node[note,text width=4.65cm] at (10.95,10.15)
  {At a good history,\\$s<T$, $Y_s\ge z_{\alpha_*}$,\\[3pt]
   $u_s(x)\ge\lambda[1-b(x_1-H_s)]$\\for $x_1>H_s$.};
\fill[PMgreen] (11.05,8.03) circle[radius=.065];
\node[anchor=west,font=\footnotesize,align=left] at (11.3,8.03)
  {Particle with no jumps\\in $(s,s+h]$};
\draw[PMorange,fill=white,line width=1pt]
  (11.05,7.3) circle[radius=.065];
\node[anchor=west,font=\footnotesize,align=left] at (11.3,7.3)
  {Additional particle at $x_1$\\with a prescribed first jump};
\node[note,font=\footnotesize,text width=4.65cm] at (10.95,6.68)
  {$U\in\partial^o\mathcal E_s$ is vacant.\\
   $x_1$ stays occupied at the jump.};
\node[anchor=west,font=\footnotesize] at (.25,6.34)
  {$q_j=1-b(j),\qquad p_M=1-\exp\{-\lambda q_Me^{-2h}\}>p_c$};
\node[anchor=west,font=\footnotesize] at (.25,5.89)
  {$\mathbb P(T\le s+h\mid\mathcal G_s)\ge p>0$ on the good-history event.};

\draw[black!18,line width=.4pt] (0,5.48)--(15.6,5.48);

\node[anchor=west,font=\small\bfseries] at (0,5.1)
  {(b) Why finite engulfments cannot accumulate before contact};
\begin{scope}[shift={(3.5,2.55)},x=.55cm,y=.55cm]
  \fill[PMblue!5]
    (-4,-2)--(-3,-2)--(-3,-3)--(2,-3)--(2,-2)--(4,-2)
    --(4,2)--(3,2)--(3,3)--(-2,3)--(-2,2)--(-4,2)--cycle;
  \draw[lattice,step=1] (-4,-3) grid (4,3);
  \draw[stationary,line width=1.4pt]
    (-4,-2)--(-3,-2)--(-3,-3)--(2,-3)--(2,-2)--(4,-2)
    --(4,2)--(3,2)--(3,3)--(-2,3)--(-2,2)--(-4,2)--cycle;
  \foreach \x/\y in {-4/-2,-3/-2,-3/-3,-2/-3,-1/-3,0/-3,1/-3,2/-3,
    2/-2,3/-2,4/-2,4/-1,4/0,4/1,4/2,3/2,3/3,2/3,1/3,0/3,
    -1/3,-2/3,-2/2,-3/2,-4/2,-4/1,-4/0,-4/-1}
    \fill[PMgreen] (\x,\y) circle[radius=.085];
  \draw[stationary,->] (-4,1)--(-5,1)--(-5,2)--(-6,2);
  \draw[stationary,->] (4,1)--(5,1)--(5,2)--(6,2);
  \draw[stationary,->] (1,-3)--(1,-4)--(3,-4);
  \node[anchor=east,text=PMgreen] at (-6,2) {$\infty$};
  \node[anchor=west,text=PMgreen] at (6,2) {$\infty$};
  \foreach \x/\y in {-1/0,0/0,0/1,1/0}
    \path[fill=PMblue!80,draw=white,line width=.5pt]
      (\x-.44,\y-.44) rectangle (\x+.44,\y+.44);
  \draw[PMblue,dashed,line width=.8pt]
    (-1.5,-1.5)--(1.5,-1.5)--(1.5,-.5)--(2.5,-.5)
    --(2.5,1.5)--(-.5,1.5)--(-.5,.5)--(-1.5,.5)--cycle;
  \node[font=\small,text=white] at (0,0) {$\mathcal E_a$};
  \node[font=\small,text=PMblue] at (2,1) {$\mathcal E_t$};
  \node[font=\normalsize,text=PMblue] at (-2.65,-.6) {$D_a$};
  \node[font=\normalsize,text=PMgreen,fill=white,inner sep=2pt]
    at (3,3.1) {$\mathcal I_a$};
\end{scope}

\node[note,text width=7.95cm] at (7.65,4.6)
  {$1-\exp\{-\lambda e^{-2h_0}\}>p_c$, $\quad a=nh_0<T$.\\[3pt]
   $\mathcal I_a$: infinite component formed by particles with no jumps
   in $(a,a+h_0]$.};
\node[note,text width=7.95cm] at (7.65,3.24)
  {$D_a$ is the finite component of
   $\mathbb Z^2\setminus\mathcal I_a$ containing $\mathcal E_a$.};
\node[note,text width=7.95cm] at (7.65,2.35)
  {Before contact in $(a,a+h_0]$, $\mathcal E_t\subset D_a$.\\
   At most $|D_a\setminus\mathcal E_a|$ finite updates can occur.};
\node[note,text width=7.95cm] at (7.65,1.2)
  {First contact engulfs $\mathcal I_a$.\\
   Thus finite updates cannot accumulate first.};
\node[anchor=west,font=\footnotesize] at (.25,.07)
  {A finite Poisson completion and finite-modification stability give $\mathcal I_a$ and $|D_a|<\infty$.};
\end{tikzpicture}
\endgroup
\caption{The two parts of the proof of
Theorem \ref{thm:supercritical}.
(a) Stationary particles connect $x_1$ to an infinite occupied
component in the forward half-plane. An additional particle jumps
from $x_1$ to the vacant boundary site $U$, forcing explosion unless
it has already occurred.
(b) During $(a,a+h_0]$, the aggregate remains in the finite component
$D_a$ until it contacts the stationary infinite cluster
$\mathcal I_a$. Each finite update adds a vertex of $D_a$, whereas
contact engulfs $\mathcal I_a$. Thus finite updates cannot accumulate
before an infinite engulfment.}
\label{fig:supercritical-proof}
\end{figure}

\begin{proof}[Proof of Theorem \ref{thm:supercritical}]
Fix $\lambda>\lambda_c$. Since $\lambda_c>2/3$, choose $\alpha_*>0$
so that Lemmas \ref{lem:good} and \ref{lem:profile-density} apply.
Put $q_j=1-b(j)$ for $j\ge1$.
Then $q_1>0$ and $q_j\uparrow1$.
Choose a deterministic $M\ge1$ with $\lambda q_M>\lambda_c$,
and then $h>0$ such that
$p_M=1-\exp\{-\lambda q_Me^{-2h}\}>p_c$.

Fix a deterministic $s\ge0$ and work on
$\{s<T,\ Y_s\ge z_{\alpha_*}\}$.
Choose the lexicographically first vertex $U\in K_s$ with $U_1=H_s$,
and put $x_j=U+je_1$.
Since $U$ is rightmost in $K_s$, it belongs to $\partial^o\Ecal_s$.

For each $x$, let $N^0_{s,h}(x)$ count particles free at $x$ at time $s$
whose prescribed walks make no jump in $(s,s+h]$.
Let $N^-_{s,h}(x)$ count those whose first jump after $s$ occurs by
$s+h$ and has direction $-e_1$.
Conditional on $\Gcal_s$, all these counts are independent Poisson
variables, with respective means
$u_s(x)e^{-2h}$ and $u_s(x)(1-e^{-2h})/4$;
see Lemma \ref{lem:density} and \cite[Theorem 5.8]{LP}.

Write $S_{s,h}=\{x:N^0_{s,h}(x)\ge1\}$.
By Lemma \ref{lem:profile-density}, this set dominates Bernoulli
site percolation of parameter $p_M$ in
$\{x:x_1\ge H_s+M\}$.
Let $F_s$ be the event that $x_M$ belongs to an infinite component of
$S_{s,h}\cap\{x:x_1\ge H_s+M\}$, that
$x_1,\ldots,x_{M-1}\in S_{s,h}$, and that
$N^-_{s,h}(x_1)\ge1$.
The three requirements involve independent Poisson counts.
Therefore, on the good-history event,
\begin{equation}\label{eq:super-trial}
\Pp(F_s\mid\Gcal_s)\ge
\theta_H(p_M)
\left[1-e^{-\lambda q_1(1-e^{-2h})/4}\right]
\prod_{j=1}^{M-1}\left[1-e^{-\lambda q_je^{-2h}}\right]
=:p>0.
\end{equation}
For $M=1$, the product is $1$.

We claim that $F_s$ forces $T\le s+h$.
Suppose instead that $F_s$ occurs and $T>s+h$.
The segment $x_1,\ldots,x_M$ and the infinite component containing
$x_M$ form an infinite connected set $S\subset S_{s,h}$.
No vertex of $S$ can be absorbed by time $s+h$: at the first such
update, all its stationary particles would still be free, so the
engulfed component would contain $S$ and be infinite.
Since $x_1\in S$ is adjacent to $U$, the same argument prevents $U$
from being absorbed. Thus $U$ remains on the outer boundary.
A particle counted by $N^-_{s,h}(x_1)$ stays at $x_1$ until its first
jump and cannot be absorbed before that jump, since this would also
absorb the stationary particle at $x_1$.
Its jump into $U$ therefore engulfs $S$, a contradiction.
Consequently, \eqref{eq:super-trial} and Lemma \ref{lem:repetition}
give $T<\infty$ almost surely.

It remains to exclude accumulation before the first infinite
engulfment, whose time is denoted by $T_I$.
We will show that, during each short deterministic interval,
particles that do not move form an infinite component with only
finite complementary components. The aggregate either engulfs this
infinite component or stays inside one of the finite complementary
components, where only finitely many updates are possible.

The same $h$ as above satisfies
$p_0=1-\exp\{-\lambda e^{-2h}\}>p_c$.
Fix a deterministic $a=nh$, $n\ge0$, and work on $\{a<T\}$.
Recall that $S_{a,h}$ consists of the sites occupied at time $a$
by free particles whose prescribed walks make no jump in $(a,a+h]$.
We first verify that $S_{a,h}$ has the required percolation property.

Conditional on $\Gcal_a$, its particle counts $N^0_{a,h}(x)$ are
independent Poisson variables with means $u_a(x)e^{-2h}$.
For comparison, take independent Poisson variables $Z_a(x)$ with means
$(\lambda-u_a(x))e^{-2h}$, independent of these counts conditional
on $\Gcal_a$. By \eqref{eq:deficit},
\[
\Ee\left[\sum_x Z_a(x)\middle|\Gcal_a\right]
=e^{-2h}\sum_x(\lambda-u_a(x))
\le\lambda e^{-2h}|K_a|(1+2a)<\infty.
\]
Hence $\sum_xZ_a(x)<\infty$ almost surely.
The occupied set $\{x:N^0_{a,h}(x)+Z_a(x)\ge1\}$ is Bernoulli
site percolation of parameter $p_0$.
It differs from $S_{a,h}$ at only finitely many sites.
By the finite-modification assertion of Lemma \ref{lem:percolation},
$S_{a,h}$ therefore has an infinite component $\mathcal I_a$ such
that every nearest-neighbor component of
$\Z^2\setminus\mathcal I_a$ is finite, almost surely on $\{a<T\}$.

Since $S_{a,h}\cap\Ecal_a=\emptyset$ and $\Ecal_a$ is connected,
it lies in a single finite component $D_a$ of
$\Z^2\setminus\mathcal I_a$.
Suppose an update in $(a,a+h]$ adds a vertex of $\mathcal I_a$,
and consider the first such update.
Immediately beforehand, every vertex of $\mathcal I_a$ is still
occupied by a free particle counted by $N^0_{a,h}$: these particles
have not moved, and none has been absorbed.
The occupied component engulfed at that update therefore contains
all of $\mathcal I_a$ and is infinite.

Consequently, on $\{T<T_I,\ a<T\le a+h\}$, the aggregate avoids
$\mathcal I_a$ for every $a\le t<T$.
As it is connected and contains $\Ecal_a$, it remains inside the
same finite set $D_a$ throughout this interval.
Every update adds at least one vertex, so at most
$|D_a\setminus\Ecal_a|$ updates can occur after $a$ and before $T$.
This contradicts accumulation at $T$, proving
$\Pp(T<T_I,\ a<T\le a+h)=0$.

On $\{T<T_I\}$, the initial engulfment is finite and the first
holding time is positive, so $T>0$.
Since $T<\infty$ almost surely, taking the union over
$a=nh$, $n\ge0$, gives $\Pp(T<T_I)=0$.
Therefore $T=T_I<\infty$ almost surely.
\end{proof}


\section{Explosion at criticality}\label{sec:critical}

At $\lambda_c$, the free field has no infinite occupied component at
any finite-index update. Explosion must therefore occur through an
accumulation of finite updates. A planar contact estimate will give a
uniform chance of this accumulation from the good histories constructed
in Section \ref{sec:linear}.

\begin{lemma}[Finite updates at criticality]\label{lem:finite-update}
If $\lambda\le\lambda_c$, every finite-index engulfment is finite
almost surely, and $T_I=\infty$.
The same holds for any finite initial aggregate with independent
Poisson free counts of means at most $\lambda_c$ and a vacant buffer.
\end{lemma}

\begin{proof}
Initially the occupied field is dominated by critical site percolation.
The seed touches only finitely many of its finite clusters.
At an entrance, Lemma \ref{lem:density} couples the independent Palm
background below critical percolation. Opening the triggering vertex
joins at most four finite background clusters. Induction over the
countably many finite update indices proves the claim.
\end{proof}

\subsection{Critical contacts}

The next estimate is uniform over the finite connected obstacle.

\begin{lemma}[Critical contacts]\label{lem:contacts}
There are universal $c,\delta,C_0>0$ as follows.
Let $F$ be finite and nearest-neighbor connected, let $v$ be rightmost
in $F$, and suppose the horizontal span of $F$ is at least $10r$.
Here the horizontal span is $\max_{x\in F}x_1-\min_{x\in F}x_1$.
Put $I=\{z\notin F:d_\infty(z,F)\le C_0\}\cap\Lambda_{5r}(v)$.
Let $A_z$ be the event of an open connection from $z$ to
$\{x:x_1\ge v_1+r-C_0\}$ inside $F^c\cap\Lambda_{5r}(v)$.
For sufficiently large $r$,
\begin{equation}\label{eq:contacts}
 \sum_{z\in I}\Pp_{p_c}(A_z)\ge cr^\delta .
\end{equation}
\end{lemma}

\begin{figure}
    \centering
    \begin{tikzpicture}[
  x=1cm,y=1cm,
  font=\fontsize{10}{12}\selectfont,
  line cap=round,line join=round,
  every node/.style={inner sep=2pt},
  note/.style={font=\fontsize{9}{11}\selectfont},
  smallnote/.style={font=\fontsize{8}{10}\selectfont},
  title/.style={font=\bfseries\fontsize{11}{13}\selectfont,anchor=west},
  leader/.style={draw=black!50,line width=.45pt},
  openpath/.style={draw=contactOpen,line width=1.15pt},
  closedpath/.style={draw=contactClosed,line width=1.05pt},
  boundary/.style={draw=contactBoundary,line width=.75pt}
]
\definecolor{contactOpen}{HTML}{167956}
\definecolor{contactClosed}{HTML}{B85B2D}
\definecolor{contactBoundary}{HTML}{286D9B}
\definecolor{contactGray}{HTML}{6D7680}

\node[title] at (0,4.35) {(a) Two circuits around the obstacle};
\node[title] at (9.15,4.35) {(b) Open and closed paths near $\Gamma$};

\begin{scope}[shift={(4.2,0)},x=.70cm,y=.70cm]
  \path[fill=contactClosed!8,even odd rule]
    (-4,-4) rectangle (4,4) (-3,-3) rectangle (3,3);
  \path[fill=contactOpen!9,even odd rule]
    (-2,-2) rectangle (2,2) (-1,-1) rectangle (1,1);
  \foreach \R in {1,2,3,4} {
    \draw[black!25,line width=.4pt]
      (-\R,-\R) rectangle (\R,\R);
  }
  \draw[black!55,dash pattern=on 4pt off 3pt,line width=.6pt]
    (-5,-5) rectangle (5,5);

  \path[fill=contactGray!30,draw=contactGray!90,line width=.7pt]
    (-6,-.23) -- (-5.5,-.23) -- (-5.5,-.15)
    -- (-4.2,-.15) -- (-4.2,-.30) -- (-3.3,-.30)
    -- (-3.3,-.15) -- (-1.8,-.15) -- (-1.8,-.05)
    -- (-.7,-.05) -- (-.7,-.15) -- (-.25,-.15)
    -- (0,0) -- (-.25,.17) -- (-1.1,.12)
    -- (-1.1,.30) -- (-2,.30) -- (-2,.18)
    -- (-3.8,.18) -- (-3.8,.40) -- (-4.8,.40)
    -- (-4.8,.12) -- (-6,.12) -- cycle;
  \node[anchor=south,note] at (-4.65,.45) {$F$};
  \draw[->,black!70,line width=.6pt]
    (-4.4,.90) -- (-5.78,.90);
  \node[anchor=south,smallnote,align=center] at (-4.80,1.00)
    {continues to\\$x_1\le v_1-10r$};

  \draw[openpath]
    (1.5,0) -- (1.5,-.8) -- (1.68,-.8) -- (1.68,-1.42)
    -- (1.28,-1.42) -- (1.28,-1.65) -- (-.9,-1.65)
    -- (-.9,-1.47) -- (-1.58,-1.47) -- (-1.58,-.7)
    -- (-1.42,-.7) -- (-1.42,.85) -- (-1.62,.85)
    -- (-1.62,1.44) -- (-.55,1.44) -- (-.55,1.65)
    -- (.85,1.65) -- (.85,1.50) -- (1.48,1.50)
    -- (1.48,.5) -- (1.5,.5) -- cycle;

  \draw[closedpath]
    (3.48,0) -- (3.48,-1.55) -- (3.65,-1.72)
    -- (3.65,-3.12) -- (3.22,-3.55) -- (2.30,-3.55)
    -- (2.15,-3.70) -- (.9,-3.70) -- (.65,-3.45)
    -- (-1.45,-3.45) -- (-1.55,-3.55) -- (-3.20,-3.55)
    -- (-3.60,-3.15) -- (-3.60,-1.70)
    -- (-3.45,-1.55) -- (-3.45,-.90) -- (-3.65,-.70)
    -- (-3.65,.70) -- (-3.50,.85) -- (-3.50,1.50)
    -- (-3.70,1.70) -- (-3.70,3.0) -- (-3.20,3.50)
    -- (-2.1,3.50) -- (-1.9,3.70) -- (-.8,3.70)
    -- (-.55,3.45) -- (1.50,3.45) -- (1.65,3.60)
    -- (3.10,3.60) -- (3.60,3.10) -- (3.60,1.70)
    -- (3.45,1.55) -- (3.45,.60) -- (3.48,.57) -- cycle;

  \draw[boundary]
    (2.55,0)
    .. controls (2.65,-1.25) and (1.50,-2.65) .. (0,-2.50)
    .. controls (-1.60,-2.65) and (-2.48,-1.85) .. (-2.58,-.45)
    .. controls (-2.63,.15) and (-2.65,1.20) .. (-1.95,1.95)
    .. controls (-.85,2.8) and (1.40,2.70) .. (2.35,1.35)
    .. controls (2.50,.95) and (2.55,.45) .. (2.55,0);

  \draw[contactBoundary,line width=1.45pt]
    (2.55,0)
    .. controls (2.65,-1.25) and (1.50,-2.65) .. (0,-2.50)
    .. controls (-1.60,-2.65) and (-2.48,-1.85) .. (-2.58,-.45);
  \draw[->,contactBoundary,line width=1.05pt]
    (.45,-2.47) .. controls (.17,-2.52) and (-.10,-2.52) .. (-.40,-2.50);

  \draw[->,black!65,line width=.55pt] (0,0) -- (5.42,0);
  \fill[black] (0,0) circle (1.8pt);
  \node[anchor=north east,note] at (-.06,-.08) {$v$};
  \fill[contactBoundary] (2.55,0) circle (1.8pt);
  \node[anchor=south west,note,text=contactBoundary,fill=white]
    at (2.62,.14) {$q_0$};
  \node[anchor=south,smallnote,fill=white] at (4.50,.14) {$v+se_1$};

  \fill[contactBoundary] (-2.58,-.45) circle (1.8pt);
  \node[anchor=north east,note,text=contactBoundary]
    at (-2.74,-.57) {$q_1$};
  \draw[contactBoundary!65,dash pattern=on 2pt off 2pt,line width=.5pt]
    (-2.90,-.77) rectangle (-2.27,-.12);
  \node[anchor=south,note,text=contactBoundary,fill=white]
    at (.15,2.62) {$\Gamma$};

  \node[anchor=north west,smallnote,fill=white,inner sep=1pt]
    at (1.04,-1.04) {$\Lambda_r(v)$};
  \node[anchor=north west,smallnote,fill=white,inner sep=1pt]
    at (2.04,-2.04) {$\Lambda_{2r}(v)$};
  \node[anchor=north west,smallnote,fill=white,inner sep=1pt]
    at (3.04,-3.04) {$\Lambda_{3r}(v)$};
  \node[anchor=north west,smallnote,fill=white,inner sep=1pt]
    at (4.04,-4.04) {$\Lambda_{4r}(v)$};
  \node[anchor=north west,smallnote,fill=white,inner sep=1pt]
    at (5.04,-5.04) {$\Lambda_{5r}(v)$};
\end{scope}

\draw[openpath] (.10,-4.08) -- (.65,-4.08);
\node[anchor=west,note] at (.75,-4.08) {open circuit};
\draw[closedpath] (3.02,-4.08) -- (3.57,-4.08);
\node[anchor=west,note] at (3.67,-4.08) {closed circuit};
\draw[boundary,line width=1.2pt] (6.13,-4.08) -- (6.68,-4.08);
\node[anchor=west,note] at (6.78,-4.08) {$\Gamma=\partial_{\rm ext}K$};

\begin{scope}[shift={(9.3,1.95)}]
  \path[fill=contactGray!30,draw=contactGray!90,line width=.7pt]
    (-.10,-1.00) -- (.43,-1.10) -- (.30,-.75)
    -- (.60,-.55) -- (.60,.45) -- (.38,.70)
    -- (.58,1.20) -- (.40,1.55) -- (-.10,1.60) -- cycle;
  \node at (.17,.93) {$F$};

  \draw[contactClosed!75,dash pattern=on 2pt off 2pt,line width=.6pt]
    (1.10,-.32) rectangle (2.10,.68);
  \node[anchor=south,note,text=contactClosed] at (1.60,.75)
    {$\Lambda_5(z)$};

  \draw[openpath]
    (1.60,.18) -- (2.42,.18) -- (2.42,.53)
    -- (3.68,.53) -- (3.68,.28) -- (4.32,.28)
    -- (4.32,.68) -- (5.38,.68) -- (5.38,.38)
    -- (6.38,.38) -- (6.38,.18) -- (7.20,.18);
  \draw[boundary,line width=.9pt]
    (1.60,0) -- (2.60,0) -- (2.60,.35)
    -- (3.50,.35) -- (3.50,.10) -- (4.50,.10)
    -- (4.50,.50) -- (5.20,.50) -- (5.20,.20)
    -- (6.20,.20) -- (6.20,0) -- (7.20,0);
  \draw[->,contactBoundary,line width=.9pt] (4.25,.10) -- (3.80,.10);
  \draw[closedpath]
    (1.60,-.18) -- (2.60,-.18) -- (2.78,0)
    -- (2.78,.17) -- (3.32,.17) -- (3.32,-.08)
    -- (4.50,-.08) -- (4.68,.10) -- (4.68,.32)
    -- (5.02,.32) -- (5.02,.02) -- (6.02,.02)
    -- (6.02,-.18) -- (7.20,-.18);

  \fill[contactOpen] (1.60,.18) circle (2pt);
  \node[anchor=south west,note,text=contactOpen] at (2.22,.76) {$z\in I$};
  \draw[leader,draw=contactOpen!65] (2.20,.77) -- (1.64,.24);
  \fill[contactBoundary] (1.60,0) circle (1.4pt);
  \node[anchor=north east,note,text=contactBoundary] at (1.53,-.05) {$q_1$};
  \fill[contactBoundary] (7.20,0) circle (1.5pt);
  \node[anchor=west,note,text=contactBoundary] at (7.28,0) {$q_0$};

  \node[note,text=contactOpen,anchor=south] at (5.1,1.03)
    {open path in $A_z$};
  \node[note,text=contactClosed,anchor=north] at (5.05,-.33)
    {closed path in $B_z$};
  \draw[<->,black!65,line width=.5pt] (.60,-.64) -- (1.60,-.64)
    node[midway,below,smallnote] {$10$};
  \node[note,text=contactBoundary] at (4.55,-.94)
    {follow $\Gamma$ from $q_0$ towards $F$};
  \draw[<->,black!60,line width=.5pt] (1.60,-1.43) -- (7.20,-1.43);
  \node[anchor=north,smallnote,align=center] at (4.40,-1.50)
    {the closed path reaches distance at least $r/2$ from $z$};
\end{scope}

\node[anchor=north west,note,align=left,text width=8.0cm] at (9.15,-.13)
  {$A_z$: an open connection in $F^c\cap\Lambda_{5r}(v)$\\[2pt]
   from $z$ to $\{x:x_1\ge v_1+r-C_0\}$.\\[8pt]
   $B_z$: a closed path from $\Lambda_5(z)$ to
   $\Lambda_{r/2}(z)^c$;\\[2pt]
   it may pass through $F$, and $\mathbb P(B_z)\le Cr^{-\delta}$.\\[8pt]
   $I=\{z\notin F:d_\infty(z,F)\le C_0\}\cap\Lambda_{5r}(v)$.};

\node[anchor=north west,note,align=left,text width=8.0cm] at (9.15,-2.88)
  {$K$: the open component of the inner circuit.\\[3pt]
   $\Gamma$: its outer boundary.\\[3pt]
   $q_1$: the first point with $d_\infty(q_1,F)=10$.\\[3pt]
   The nearby open path avoids $F$.};

\begin{scope}[yshift=-.45cm]
\draw[black!18,line width=.5pt] (0,-4.57) -- (17.4,-4.57);
\node[anchor=west,note] at (.10,-4.94)
  {$G$: both circuits in (a) exist.\quad
   $\mathbb P(G)\ge c>0$,\qquad
   $G\subseteq\bigcup_{z\in I}(A_z\cap B_z)$.};
\node at (8.70,-5.77)
  {$\displaystyle
   c\le\sum_{z\in I}\mathbb P(A_z\cap B_z)
   \le\sum_{z\in I}\mathbb P(A_z)\mathbb P(B_z)
   \le Cr^{-\delta}\sum_{z\in I}\mathbb P(A_z).$};
\node[note] at (8.70,-6.56)
  {$A_z$ is increasing and $B_z$ is decreasing: Harris's inequality gives
   $\sum_{z\in I}\mathbb P(A_z)\ge c'r^\delta$.};
\node[anchor=west,smallnote,text=black!60,align=left] at (.10,-7.25)
  {All probabilities are at $p_c$. Open paths use nearest-neighbor steps;
   closed paths may also use diagonal steps.\\[2pt]
   Grey denotes the deterministic set $F$, not its site states.
   The drawing is schematic; fixed distances and path separation are enlarged.};
\end{scope}
\end{tikzpicture}
    \caption{Illustration for the proof of Lemma \ref{lem:contacts}}
    \label{fig:lemma_6_2}
\end{figure}

\begin{proof}
We will find, with probability bounded below, an open path counted
by some event $A_z$ and a long closed path starting near the same vertex $z$.
Harris's inequality will then give the result.

Throughout the proof, open paths use nearest-neighbor steps,
whereas closed paths may also use diagonal steps.
Sample independent site percolation of parameter $p_c$ on the
whole lattice, including $F$.

Let $G$ be the event that there is an open circuit surrounding $v$
in $\Lambda_{2r}(v)\setminus\Lambda_r(v)$ and a closed circuit
surrounding $v$ in
$\Lambda_{4r}(v)\setminus\Lambda_{3r}(v)$.
The two events depend on disjoint sets of vertices, so
Lemma \ref{lem:percolation} gives $\Pp_{p_c}(G)\ge c>0$.

On $G$, choose such an open circuit and let $K$ be its open
connected component. The outer closed circuit prevents $K$ from
leaving $\Lambda_{4r}(v)$.
We describe how to trace the outer boundary of $K$.
Place a unit square centred at each vertex of $K$ and follow the
boundary adjacent to the unbounded complementary region.
At a corner where two diagonally opposite squares belong to $K$,
separate the two boundary passages by a small displacement,
keeping the two open squares separate at that corner.
This gives a simple closed curve $\Gamma$ surrounding the chosen
open circuit. For the correspondence between planar boundary interfaces and right-most
paths, see Biskup, Louidor, Procaccia and Rosenthal
\cite[Section~2.2, Proposition~2.3]{BLPR}.
See also the construction of surrounding loops by Cai, Procaccia and Zhang
\cite[Section~7.1, Lemma~7.5 and Remark~7.6]{CPZ}.
For our square-site model, we specify the local convention below to ensure
that the open path is nearest-neighbor connected and the closed path is
connected in the matching lattice.

Every segment of $\Gamma$ determines two paths within distance
$2$ of that segment. On the inside, the centres of the adjacent
squares give an open nearest-neighbor path: whenever two successive
centres are diagonal neighbors, insert the centre of the square
between them on the inside, which also belongs to $K$.
On the outside, the adjacent square centres are closed, since an
open neighbor of $K$ would belong to $K$.
Successive outside centres are nearest or diagonal neighbors,
so they give a closed path.

Choose a point $q_0$ where $\Gamma$ meets the horizontal ray
$\{v+se_1:s\ge0\}$.
Since $\Gamma$ surrounds the inner open circuit,
$(q_0)_1\ge v_1+r-1$.
Every vertex of $F$ has first coordinate at most $v_1$, so
$q_0$ is at distance at least $r-1$ from $F$.

On the other hand, connectedness and the horizontal span of $F$
give a nearest-neighbor path in $F$ from $v$ to outside
$\Lambda_{5r}(v)$.
Its piecewise linear interpolation must cross $\Gamma$.
Thus $\Gamma$ comes within distance $1$ of $F$.
Starting at $q_0$, follow $\Gamma$ clockwise until the first point
$q_1$ at distance $10$ from $F$, using $\ell^\infty$ distance.
Every point of this segment is at distance at least $10$ from $F$.

The open path obtained from this segment therefore avoids $F$.
Its endpoint $z$ near $q_1$ satisfies
$d_\infty(z,F)\le12$, and its endpoint near $q_0$ has first
coordinate at least $v_1+r-3$.
Both paths lie in $\Lambda_{4r+3}(v)\subset\Lambda_{5r}(v)$.
Taking $C_0\ge20$, we have $z\in I$ and $A_z$ occurs.

The closed path starts within distance $4$ of $z$ and reaches
a vertex whose first coordinate is at least $v_1+r-3$.
Since $z_1\le v_1+12$, this path reaches distance at least $r/2$
from $z$ when $r$ is sufficiently large.

For each deterministic $z\in I$, let $B_z$ be the event that
a closed path connects $\Lambda_5(z)$ to the complement of
$\Lambda_{r/2}(z)$.
The path in this definition may pass through $F$.
The construction above proves
$G\subseteq\bigcup_{z\in I}(A_z\cap B_z)$.
Moreover, Lemma \ref{lem:percolation} gives
$\Pp_{p_c}(B_z)\le Cr^{-\delta}$ uniformly in $z$.

Each $A_z$ is increasing and each $B_z$ is decreasing.
Consequently, the union bound and Harris's inequality give
\[
c
\le \Pp_{p_c}(G)
\le \sum_{z\in I}\Pp_{p_c}(A_z\cap B_z)
\le \sum_{z\in I}\Pp_{p_c}(A_z)\Pp_{p_c}(B_z)
\le Cr^{-\delta}\sum_{z\in I}\Pp_{p_c}(A_z).
\]
Although the construction uses the auxiliary states on $F$,
each $A_z$ depends only on the states in $F^c$.
Rearranging proves \eqref{eq:contacts}.
\end{proof}

\subsection{A uniform explosion trial}

\begin{lemma}[Explosion from a uniform forward density bound]
\label{lem:restart}
Let $b:\mathbb N\to[0,1]$ be nonincreasing, with $b(1)<1$ and
$b(d)\le C_m d^{-m}$ for every $m>0$.
Let $E\subset\Z^2$ be finite, nonempty and nearest-neighbor connected.
Put $K=E\cup\partial^oE$ and $H=\max_{x\in K}x_1$.

Start PMDLA from $E$, with independent Poisson free-particle counts
of means $0\le\mu(x)\le\lambda_c$, where $\mu=0$ on $K$ and
\begin{equation}\label{eq:trial-initial-density}
\mu(x)\ge\lambda_c[1-b(x_1-H)]
\qquad\text{for every }x_1>H.
\end{equation}
The bound is uniform in the second coordinate.
There is $p_*>0$, depending only on $b$,
such that $\Pp(T<1)\ge p_*$.
\end{lemma}

\begin{remark}The hypothesis \eqref{eq:trial-initial-density} is supplied by
Lemma \ref{lem:profile-density} at a good history.
It is a quantitative density bound; asymptotic linear speed alone
does not imply the lemma.
\end{remark}
Figure \ref{fig:restartlemma} summarizes the proof notation and ideas.
\begin{figure}
    \centering
    \begin{tikzpicture}[
  x=1cm,y=1cm,
  font=\fontsize{10}{12}\selectfont,
  >=Stealth,
  line cap=round,line join=round,
  every node/.style={inner sep=2pt},
  note/.style={font=\fontsize{9}{11}\selectfont},
  smallnote/.style={font=\fontsize{9}{11}\selectfont},
  point/.style={circle,fill=black,inner sep=1.8pt},
  contact/.style={circle,draw=pmdlaContact,fill=white,line width=.85pt,
    inner sep=2pt},
  leader/.style={draw=black!55,line width=.45pt},
  openpath/.style={draw=pmdlaOpen,line width=1.45pt},
  bridge/.style={draw=pmdlaBridge,line width=1.35pt}
]
\definecolor{pmdlaBlue}{HTML}{286888}
\definecolor{pmdlaLightBlue}{HTML}{EAF4FA}
\definecolor{pmdlaBuffer}{HTML}{FBF0DB}
\definecolor{pmdlaBridge}{HTML}{B97820}
\definecolor{pmdlaGray}{HTML}{68737E}
\definecolor{pmdlaOpen}{HTML}{197B61}
\definecolor{pmdlaContact}{HTML}{A33962}

\node[anchor=west,font=\bfseries\fontsize{11}{13}\selectfont]
  at (0,3.65) {(a) Contacts near the current rightmost point};

\fill[pmdlaLightBlue] (9.55,-2.25) rectangle (14.05,2.25);
\fill[pmdlaOpen!7] (12.16,-2.25) rectangle (14.05,2.25);

\path[fill=pmdlaBuffer,draw=pmdlaBridge!85,
  line width=.8pt,dash pattern=on 4pt off 2pt]
  (.08,-.6) -- (.8,-1.6) -- (1.9,-1.35) -- (3.1,-1.45)
  -- (4.1,-1.08) -- (5.3,-1.3) -- (6.2,-1.1)
  -- (7.0,-1.6) -- (8.05,-1.18) -- (8.8,-1.03)
  -- (9.4,-.94) -- (10.55,-.94) -- (11.2,-.6)
  -- (11.8,0) -- (11.2,.65) -- (10.55,.9) -- (9.3,.9)
  -- (8.8,1.15) -- (8.0,1.52) -- (7.1,1.25)
  -- (6.0,1.6) -- (5.2,1.25) -- (4.2,1.45)
  -- (3.2,1.12) -- (2.4,1.4) -- (1.4,1.14)
  -- (.4,1.25) -- (.08,.55) -- cycle;

\path[fill=pmdlaGray!35,draw=pmdlaGray!90,line width=.8pt]
  (.6,-.45) -- (1.25,-.95) -- (2.05,-.7) -- (3.2,-.8)
  -- (4.1,-.45) -- (5.2,-.67) -- (6.2,-.48)
  -- (7.05,-.98) -- (7.95,-.58) -- (8.8,-.42)
  -- (9.4,-.3) -- (10.45,-.3) -- (11.2,0)
  -- (10.6,.25) -- (9.5,.25) -- (8.95,.48)
  -- (8.0,.86) -- (7.1,.58) -- (6.0,.96)
  -- (5.2,.58) -- (4.2,.8) -- (3.2,.47)
  -- (2.4,.75) -- (1.4,.48) -- (.6,.65) -- cycle;

\draw[black!55,dash pattern=on 4pt off 3pt] (2.2,-2.5) -- (2.2,2.75);
\draw[pmdlaGray,dash pattern=on 3pt off 3pt] (11.2,-2.5) -- (11.2,2.8);
\draw[pmdlaBlue,line width=.9pt] (9.55,-2.25) rectangle (14.05,2.25);
\draw[pmdlaOpen,dash pattern=on 2pt off 2pt,line width=.8pt]
  (12.16,-2.25) -- (12.16,2.25);

\node[align=center,anchor=south] at (2.2,2.78)
  {$x_1=H$\\[-1pt]\scriptsize fixed initial line};
\node[align=center,anchor=south] at (11.2,2.83)
  {$x_1=H_t=H+V$\\[-1pt]\scriptsize current rightmost coordinate};
\node[anchor=west,text=pmdlaBlue] at (14.2,2.05) {$Q_t=\Lambda_{5r}(v_t)$};
\node[anchor=west,smallnote,text=black!65] at (14.2,1.65)
  {$r=\lfloor V/20\rfloor$};

\node at (5.1,.15) {$K_t=\mathcal E_t\cup\partial^o\mathcal E_t$};
\node[anchor=west,text=pmdlaBridge,note,fill=white,inner sep=3pt] at (.15,-1.98)
  {$F_t=\{x:d_1(x,K_t)\le L\}$};
\node[align=center,note] at (5.7,2.15)
  {Initial density, uniformly in $x_2$:\\[2pt]
   $\mu(x)\ge\lambda_c[1-b(x_1-H)],\quad x_1>H$};

\node[point,label={[smallnote,anchor=south east]140:$U_t$}]
  (Ut) at (11.2,0) {};
\node[point,fill=pmdlaBridge,
  label={[smallnote,anchor=north west,fill=white,inner sep=1.5pt]-30:$v_t=U_t+Le_1$}]
  (vt) at (11.8,0) {};
\draw[<->,black!65,line width=.5pt] (11.2,-.75) -- (11.8,-.75)
  node[midway,below,smallnote] {$L$};

\node[contact] at (9.92,1.03) {};
\node[contact] at (10.69,-1.03) {};
\node[contact] (zmain) at (11.52,.49) {};
\node[smallnote,text=pmdlaContact,anchor=south east] at (11.52,1.04) {$z\in I_t$};
\draw[leader,draw=pmdlaContact!75] (11.48,1.01) -- (11.52,.56);
\draw[openpath] (zmain) -- (11.87,.49) -- (11.87,.82)
  -- (12.04,.82) -- (12.04,.62) -- (12.55,.62)
  -- (12.55,1.02) -- (13.3,1.02);
\fill[pmdlaOpen] (13.3,1.02) circle (1.8pt);
\node[smallnote,text=pmdlaOpen,align=center,fill=pmdlaOpen!7] at (13.05,1.58)
  {open path\\in $A_{t,z}$};
\node[text=pmdlaBlue,align=center,fill=pmdlaLightBlue] at (10.98,-1.68)
  {$D_t=Q_t\setminus F_t$};

\draw[leader] (13.0,-2.96) -- (12.16,-2.48) -- (12.16,-2.25);
\node[anchor=north west,smallnote,text=pmdlaOpen,align=left] at (12.3,-2.98)
  {target:\\[-1pt]$x_1=(v_t)_1+r-C_0$};
\draw[<->,line width=.6pt] (2.2,-2.84) -- (11.2,-2.84)
  node[midway,fill=white,inner sep=3pt] {$V=H_t-H$};
\draw[leader] (8.42,-3.62) -- (9.55,-3.12) -- (9.55,-2.25);
\node[anchor=east,note] at (8.42,-3.62)
  {left edge: $H+V+L-5r\ge H+\tfrac34V+L$};

\begin{scope}[shift={(0,-6.95)}]
\node[anchor=west,font=\bfseries\fontsize{11}{13}\selectfont]
  at (0,2.35) {(b) An entrance connects the aggregate to one contact};

\fill[pmdlaLightBlue] (0,-.9) rectangle (7.8,1.05);
\fill[pmdlaBuffer] (0,-.9) rectangle (5.5,1.05);
\fill[pmdlaGray!24] (0,-.9) rectangle (2.5,1.05);
\fill[pmdlaGray!60] (0,-.9) rectangle (1.5,1.05);
\draw[pmdlaGray!80,line width=.75pt] (1.5,-.9) -- (1.5,1.05);
\draw[pmdlaGray,line width=.75pt] (2.5,-.9) -- (2.5,1.05);
\draw[pmdlaBridge,line width=.8pt,dash pattern=on 4pt off 2pt]
  (5.5,-.9) -- (5.5,1.05);

\node at (.7,1.47) {$\mathcal E_t$};
\node at (2.0,1.47) {$K_t$};
\node[text=pmdlaBridge] at (4.25,1.47) {$F_t$};
\node[text=pmdlaBlue] at (6.65,1.47) {$D_t$};

\draw[bridge] (1,0) -- (6,0);
\fill[black] (1,0) circle (2pt);
\draw[fill=white,line width=.8pt] (2,0) circle (3.3pt);
\draw[fill=white,line width=.7pt] (3,0) circle (5.5pt);
\fill[pmdlaOpen] (2.93,-.02) circle (1.9pt);
\fill[pmdlaOpen] (3.07,.02) circle (1.9pt);
\fill[pmdlaOpen] (4,0) circle (2.4pt);
\fill[pmdlaOpen] (5,0) circle (2.4pt);
\node[contact] at (6,0) {};
\fill[pmdlaOpen] (6,0) circle (1.2pt);
\draw[openpath,-{Stealth[length=2.2mm]}] (6.06,0) -- (6.5,0)
  -- (6.5,.35) -- (6.9,.35) -- (6.9,.15) -- (7.28,.15)
  -- (7.28,.6) -- (7.72,.6);

\draw[-{Stealth[length=2.1mm]},black,line width=.8pt]
  (3,.2) to[out=130,in=50] node[above,smallnote] {jump} (2,.2);
\node[smallnote,anchor=north] at (1,-.17) {$y_0$};
\node[smallnote,anchor=north,align=center] at (2,-.17) {$w=y_1$\\vacant};
\node[smallnote,anchor=north] at (3,-.17) {$y_2$};
\node[smallnote,anchor=north] at (4,-.17) {$y_3$};
\node[smallnote,anchor=north,text=pmdlaContact] at (6,-.17) {$z$};
\node[smallnote,text=pmdlaOpen] at (7.1,.92) {$A_{t,z}$};
\node[smallnote,anchor=north west,align=left] at (3.2,-.44)
  {At $y_2$: one particle jumps;\\
   at least one remains.};

\draw[<->,black!65,line width=.5pt] (1,-1.26) -- (6,-1.26)
  node[midway,below,smallnote] {connecting path of length at most $L+2C_0+1$};

\node[anchor=north west,note,align=left,text width=7.7cm] at (8.75,1.54)
  {Every contact event uses only sites in $D_t$:\\[4pt]
   $x_1-H\ge\tfrac34V+L$,\quad $d_1(x,K_t)>L$.\\[6pt]
   $\displaystyle\lambda_c-u_t(x)\le
     \lambda_c\mathbb P(\mathrm{Pois}(2)\ge L)+C_mV^{-m}.$\\[8pt]
   On the connecting path outside $K_t$:\\[3pt]
   $u_t(x)\ge m_*>0$.};

\node[anchor=north west,note,align=left,text width=7.7cm] at (8.75,-1.37)
  {Entrance at $w$ + occupied continuation + $A_{t,z}$\\[3pt]
   $\Longrightarrow\ H_{\rm new}\ge H+1.01V$.};
\end{scope}

\node[anchor=west,smallnote,text=black!65] at (0,-9.42)
  {$I_t=\{z\in D_t:d_\infty(z,F_t)\le C_0\}$.\quad
   Schematic; the lattice is suppressed in (a), and the width $L$ is exaggerated.};
\end{tikzpicture}
    \caption{Illustration of the proof of Lemma \ref{lem:restart}\label{fig:restartlemma}}
\end{figure}

\begin{proof}
Throughout the proof, $H$ denotes the fixed initial rightmost
coordinate. Write
$K_t=\Ecal_t\cup\partial^o\Ecal_t$,
$H_t=\max_{x\in K_t}x_1$, and $V_t=H_t-H$.
We will use contacts near the current rightmost point, in a region
that remains far ahead of the initial line $x_1=H$.

\smallskip
\noindent\emph{Producing an initial extension.}
Put $m_0=\lambda_c[1-b(1)]>0$ and $\delta_0=1/4$.
Choose a rightmost vertex $U$ of $K$, and set
$x_j=U+je_1$ for $1\le j\le M$.
The integer $M$ will be chosen at the end.

For each $x$, independently sample $N_x\sim\operatorname{Pois}(\mu(x))$,
and then sample $N_x$ independent rate-two random-walk trajectories
from $\Pp_x$. All trajectories are sampled for all times;
the aggregation process uses each trajectory only until its
particle is absorbed.
Let $J_t(\omega)$ denote the number of jumps of a trajectory
$\omega$ by time $t$.

For $1\le j\le M$, let $Z_j$ count the sampled trajectories in
\[
\mathcal D_j
=\{\omega:\omega(0)=x_j,\ J_{\delta_0}(\omega)=0\},
\]
and let $Z_*$ count those in
\[
\mathcal D_*
=\{\omega:\omega(0)=x_1,\ J_{\delta_0}(\omega)=1,\
                         \omega(\delta_0)=U\}.
\]
Define $F_M=\{Z_1=\cdots=Z_M=Z_*=1\}$.
There are no restrictions on the other trajectories.

Poisson thinning makes these counts independent, with means
$\alpha_j=\mu(x_j)e^{-2\delta_0}$ and
$\beta=\mu(x_1)(\delta_0/2)e^{-2\delta_0}$.
Since $m_0\le\mu(x_j)\le\lambda_c$,
\[
\Pp(F_M)
=e^{-\sum_j\alpha_j-\beta}\,\beta\prod_{j=1}^M\alpha_j
\ge
e^{-(M+1)\lambda_c}
(m_0e^{-2\delta_0})^{M+1}\frac{\delta_0}{2}
=:p_M>0.
\]

On $F_M$, the sites $x_1,\ldots,x_M$ stay occupied until the
first absorption of one of them, which absorbs the whole segment.
If this has not occurred before the specified jump from $x_1$
to $U$, then $U$ is still on the outer boundary: any earlier
absorption of $U$ would also have absorbed its occupied neighbor
$x_1$ and the rest of the segment.
The specified jump therefore forces absorption of the segment.

Thus, on $F_M\cap\{T>\delta_0\}$, all specified particles have
been absorbed by $\delta_0$, and $V_{\delta_0}\ge M$.

\smallskip
\noindent\emph{Density relative to the fixed initial line.}
Let $\mathcal D=\mathcal D_*\cup\bigcup_j\mathcal D_j$.
Let $\mathcal H$ record the number of sampled trajectories in
$\mathcal D$ and their restrictions to $[0,\delta_0]$.
For $t\ge\delta_0$, put
$\widehat{\Gcal}_t=\Gcal_t\vee\mathcal H$.
In particular, $F_M$ is measurable at time $\delta_0$.

Conditional on $\mathcal H$, Poisson thinning leaves the
trajectories outside $\mathcal D$ as an independent Poisson
process with intensity
$\sum_x\mu(x)\Pp_x(d\omega)\mathbf1_{\{\omega\notin\mathcal D\}}$.
On $F_M$, all particles belonging to $\mathcal D$ have been
absorbed by $\delta_0$, unless explosion has already occurred.
Since membership in $\mathcal D$ depends only on the trajectory
up to $\delta_0$, the remaining particles have ordinary
random-walk transition rates after this time.

Apply Lemma \ref{lem:density} to the remaining trajectories.
Conditional on $\widehat{\Gcal}_t$, the free counts are independent
Poisson variables with means $u_t(x)\le\lambda_c$.
To obtain a lower bound, retain only trajectories that avoid
the current set $K_t$ throughout $[0,t]$.
They avoid all earlier explored sets and cannot start at any
$x_j$, since these sites now belong to $K_t$.
They therefore belong to none of the deleted sets $\mathcal D_j$
or $\mathcal D_*$.
Reversibility gives, on $F_M$ and for
$\delta_0\le t\le1$ with $t<T$,
\begin{equation}\label{eq:restart-filter}
u_t(x)\ge
\Ee_x\bigl[\mu(X_t);
           X_s\notin K_t\text{ for every }0\le s\le t\bigr].
\end{equation}

For $x\notin K_t$ with $x_1>H$, the trajectory making no jump
gives $u_t(x)\ge m_*$, where $m_*=m_0e^{-2}>0$.
This is the positive density bound used close to the aggregate.

A stronger bound holds farther from the aggregate.
Write $d_1$ for graph distance.
If $d_1(x,K_t)\ge L$ and $d=x_1-H>0$, then
\begin{equation}\label{eq:restart-buffer}
\lambda_c-u_t(x)
\le\lambda_c\Pp(\operatorname{Pois}(2)\ge L)+C_m d^{-m}
\qquad\text{for every }m>0.
\end{equation}
Indeed, hitting $K_t$ requires at least $L$ jumps.
Extend $b(k)=1$ for $k\le0$.
The remaining contribution to the deficit is at most
$\lambda_c\Ee_x b(X_t\cdot e_1-H)$.
When $|(X_t-x)\cdot e_1|\le d/2$, this is bounded using
$b(\lceil d/2\rceil)$; the complementary event requires at
least $d/2$ jumps.
The assumed decay of $b$ and the Poisson tail give
\eqref{eq:restart-buffer}.

\smallskip
\noindent\emph{Localizing every contact event.}
Consider a finite pre-update state at time $t\in[\delta_0,1]$,
and write $V=V_{t-}$.
For sufficiently large $V$, set
\[
L=\left\lceil\frac{8\log V}{\log\log V}\right\rceil,
\qquad r=\lfloor V/20\rfloor,
\qquad F_t=\{x:d_1(x,K_{t-})\le L\}.
\]
Choose a rightmost vertex $U_t$ of $K_{t-}$ and put
$v_t=U_t+Le_1$.
Then $v_t$ is rightmost in $F_t$ and
$(v_t)_1=H+V+L$.

Define the domain and contact set
\[
D_t=F_t^c\cap\Lambda_{5r}(v_t),
\qquad
I_t=\{z\in D_t:d_\infty(z,F_t)\le C_0\}.
\]
For $z\in I_t$, let $A_{t,z}$ be the event of an open
nearest-neighbor connection from $z$ to
$\{x:x_1\ge(v_t)_1+r-C_0\}$ entirely within $D_t$.

These definitions localize the whole contact sum near $U_t$.
More precisely, every site used by every $A_{t,z}$ satisfies
\begin{equation}\label{eq:trial-localization}
x_1-H\ge V+L-5r\ge\frac34V+L,
\qquad d_1(x,K_{t-})>L.
\end{equation}
Some of $D_t$ may lie behind the current line $x_1=H_{t-}$,
but all of it lies far ahead of the fixed initial line $x_1=H$.
Consequently, \eqref{eq:restart-buffer} applies uniformly
throughout $D_t$.

The set $F_t$ is connected and has horizontal span at least
$V\ge10r$.
Lemma \ref{lem:contacts} therefore gives
\[
\sum_{z\in I_t}\Pp_{p_c}(A_{t,z})\ge cV^\delta.
\]
Put $q(x)=1-e^{-u_{t-}(x)}$, and let $\Pp_q$ denote the
product measure with these occupation probabilities.
Couple it to critical percolation using independent uniforms
at the sites.
Since $q(x)\le p_c$, the probability of a disagreement at
$x\in D_t$ is
\[
p_c-q(x)\le\lambda_c-u_{t-}(x)
\le\lambda_c\Pp(\operatorname{Pois}(2)\ge L)+CV^{-8},
\]
where we used \eqref{eq:trial-localization}.

Both $|D_t|$ and $|I_t|$ are at most $CV^2$.
Every $A_{t,z}$ depends only on sites in $D_t$, so a union bound
gives
\[
\sum_{z\in I_t}
\bigl[\Pp_{p_c}(A_{t,z})-\Pp_q(A_{t,z})\bigr]
\le
CV^4\bigl[\Pp(\operatorname{Pois}(2)\ge L)+V^{-8}\bigr].
\]
The right side tends to zero, since
$\Pp(\operatorname{Pois}(2)\ge L)
\le(2e/L)^L=V^{-8+o(1)}$.
Hence, for sufficiently large $V$,
\begin{equation}\label{eq:trial-contact-sum}
\sum_{z\in I_t}\Pp_q(A_{t,z})\ge cV^\delta/2.
\end{equation}

\smallskip
\noindent\emph{Connecting the contacts to the aggregate.}
For each $z\in I_t$, choose a shortest nearest-neighbor path
$(y_0,\ldots,y_\ell)$ from $\Ecal_{t-}$ to $z$.
Its length satisfies $\ell\le L+2C_0+1$.
Its first new vertex $w(z)=y_1$ is on the vacant outer boundary,
and all subsequent vertices lie outside $K_{t-}$.

Moreover, \eqref{eq:trial-localization} gives
\[
(y_i)_1-H
\ge z_1-H-\ell
\ge\frac34V-2C_0-1>0
\]
for large $V$.
Thus every vertex $y_i$, $i\ge2$, has free particle intensity at least $m_*$.
Only this positive lower bound is needed along the connecting
path; the near-critical estimate was used inside $D_t$.

Consider an entrance along $y_2\to w(z)$.
Its conditional intensity is $u_{t-}(y_2)/2\ge m_*/2$.
After removing the triggering particle from $y_2$, the remaining
counts retain their independent Poisson laws.
Indeed, for a Poisson count configuration $N$ and every
nonnegative measurable function $f$,
\[
\Ee[N(y_2)f(N-\delta_{y_2})]
=u_{t-}(y_2)\Ee f(N).
\]
The weighting by $N(y_2)$ is exactly that produced by the jump
rate $N(y_2)/2$.

Put $b_0=1-e^{-m_*}>0$.
The probability that all $y_2,\ldots,y_\ell$ remain occupied
after setting aside the triggering particle is at least
$b_0^{L+2C_0}$.
This event and $A_{t,z}$ are increasing, so Harris's inequality
bounds their joint probability below by
$b_0^{L+2C_0}\Pp_q(A_{t,z})$.
Their occurrence forces the engulfment to reach first coordinate
at least
$(v_t)_1+r-C_0>H+1.01V$.

For each entrance site $w$, retain one assigned contact maximizing
$\Pp_q(A_{t,z})$.
At most $C(L+1)^2$ contacts are assigned to $w$, since each is
within graph distance $L+2C_0$ of $w$.
Summing over distinct entrance sites and using
\eqref{eq:trial-contact-sum}, the conditional intensity $h_t$
of updates satisfying $V_t\ge1.01V_{t-}$ obeys
\begin{equation}\label{eq:macro}
h_t\ge
\frac{c\,b_0^{L+2C_0}}{(L+1)^2}
\sum_{z\in I_t}\Pp_q(A_{t,z})
\ge c_1V^\eta,
\qquad \eta=\delta/2>0,
\end{equation}
for $V\ge V_*$.
Here $b_0^{L+2C_0}/(L+1)^2=V^{-o(1)}$.
The constants are uniform over the aggregate and the revealed past.

\smallskip
\noindent\emph{Reaching infinitely many levels.}
At criticality, every individual engulfment is finite:
its component is dominated by critical percolation with one
additional open vertex.
Thus explosion can occur only through accumulation of updates.

Put $a=1.01$, take $M\ge V_*$, and work on
$F_M\cap\{T>\delta_0\}$.
Set $V_t=\infty$ after $T$ and define
\[
\sigma_n=
\inf\{t\ge\delta_0:V_t\ge a^nM\}\wedge T\wedge1.
\]
Then $\sigma_0=\delta_0$ and $\sigma_n\uparrow T\wedge1$,
because finitely many finite updates cannot reach unbounded levels.

While $a^nM\le V_t<a^{n+1}M$, every update counted in
\eqref{eq:macro} crosses the next level.
Its intensity is at least $c_n=c_1(a^nM)^\eta$.
We prove the mean waiting-time bound directly.

Stop additionally at
$\rho_k=\sigma_{n+1}\wedge
\inf\{t\ge\sigma_n:|\Ecal_t|\ge k\}$.
Let $Q_{n,k}$ count the updates satisfying $V_s\ge aV_{s-}$
in $(\sigma_n,\rho_k]$.
There is at most one such update.
Before the additional stop, the total entrance rate is bounded
by $8\lambda_c k$, so compensation gives
\[
c_n\Ee[\rho_k-\sigma_n\mid\widehat{\Gcal}_{\sigma_n}]
\le
\Ee\left[\int_{\sigma_n}^{\rho_k}h_s\,ds
          \,\middle|\,\widehat{\Gcal}_{\sigma_n}\right]
=
\Ee[Q_{n,k}\mid\widehat{\Gcal}_{\sigma_n}]
\le1.
\]
Letting $k\to\infty$ by monotone convergence yields
$\Ee[\sigma_{n+1}-\sigma_n\mid
\widehat{\Gcal}_{\sigma_n}]\le c_n^{-1}$.
Levels already passed have zero waiting time.

Summing and conditioning at $\delta_0$, we obtain
\[
\Ee[(T\wedge1)-\delta_0\mid\widehat{\Gcal}_{\delta_0}]
\le
\sum_{n=0}^\infty\frac1{c_1(a^nM)^\eta}
=
\frac{M^{-\eta}}{c_1(1-a^{-\eta})}
\]
on $F_M\cap\{T>\delta_0\}$.
Choose $M$ so that the last quantity is at most
$(1-\delta_0)/2$.
Since $(T\wedge1)-\delta_0=1-\delta_0$ on $\{T\ge1\}$,
Markov's inequality gives conditional probability at least
$1/2$ of explosion before time $1$.

Explosion by $\delta_0$ already counts as success. Therefore
\[
\Pp(T<1)
\ge
\Pp(F_M\cap\{T\le\delta_0\})
+\frac12\Pp(F_M\cap\{T>\delta_0\})
\ge p_M/2.
\]
Taking $p_*=p_M/2$ proves the lemma.
\end{proof}

\begin{samepage}
\subsection{Almost-sure explosion}

\begin{proof}[Proof of Theorem \ref{thm:critical}]
Since $\lambda_c\ge\log2>2/3$, Lemma \ref{lem:good} applies.
At a good history, Lemma \ref{lem:profile-density} gives exactly the
intensity assumption of Lemma \ref{lem:restart}. That lemma supplies
a uniform positive chance of explosion within one further time unit.
Lemma \ref{lem:repetition} therefore gives $T<\infty$ almost surely.
Lemma \ref{lem:finite-update} gives $T_I=\infty$.
\end{proof}

\end{samepage}

\section{Higher dimensions}\label{sec:dimensions}

In this section, we prove Corollary \ref{cor:dimensions}.
The Poisson exploration and subcritical estimates extend directly
to $\Z^d$. The lower bound on forward growth improves because
particles have more ways to reach the aggregate.

Throughout the proof, walks have total rate $d$: each coordinate
has rate one, and each oriented jump has rate $1/2$.
Speed bounds must therefore be divided by $d$ in the original clock.
Recall that $p_c(d)=p_c^{\rm site}(\Z^d)$,
$\lambda_c(d)=-\log(1-p_c(d))$,
$K_t=\Ecal_t\cup\partial^o\Ecal_t$ and
$H_t=\max_{x\in K_t}x_1=h_{e_1}(t)+1$.
As before, $T$ includes accumulation of finite updates, whereas
$T_I$ is the time of an infinite engulfment.
All statements about the particle field are made before $T$.

\subsection{Poisson densities and subcritical growth}
\label{subsec:dimensions-subcritical}

The conditional Poisson law and density estimates in
Lemma \ref{lem:density} hold in every dimension, as their proof
uses only Poisson exploration and symmetry of the walk kernel.
In the present clock, they give
\begin{equation}\label{eq:d-dimensional-density}
 u_t(x)=\lambda\Pp_x(X_r\notin K_{t-r},\ 0\le r\le t),
 \qquad
 \sum_x(\lambda-u_t(x))\le\lambda|K_t|(1+dt).
\end{equation}
Here $X$ is an independent rate-$d$ walk, with the aggregate history
fixed in the probability. The conditional counts are independent
Poisson variables, $0\le u_t(x)\le\lambda$, and $u_t=0$ on $K_t$.
The entrance intensity at $w\in\partial^o\Ecal_{t-}$ is
$r_t(w)=\frac12\sum_{y\sim w}u_{t-}(y)\le d\lambda$.
These statements also hold at finite stopping times before $T$.

Suppose $p=1-e^{-\lambda}<p_c(d)$. Subcritical site clusters have
exponential size tails \cite[Theorem~3]{AV}, so
$\chi_\theta=\Ee_p\left[\sum_{x\in C(0)}e^{\theta\|x\|_1}\right]<\infty$
for small $\theta\ge0$.
Put $V_\theta(F)=\sum_{x\in F}e^{\theta\|x\|_1}$ and
$m_\theta=1+2de^\theta\chi_\theta$.
The entrance comparison in the proof of Theorem \ref{thm:upper}
bounds the mean increment at $w$ by $m_\theta e^{\theta\|w\|_1}$.
Since $\sum_{w\in\partial^o F}e^{\theta\|w\|_1}
\le2de^\theta V_\theta(F)$, the drift bound becomes
$b_\theta V_\theta(F)$ with $b_\theta=2d^2\lambda e^\theta m_\theta$.
The stopped expectation argument of that theorem gives nonexplosion
and $\Ee V_\theta(\Ecal_t)\le m_\theta e^{b_\theta t}$.
A parallel exponential Markov estimate and Borel--Cantelli argument
give the linear upper bound in part (1) of Corollary \ref{cor:dimensions},
after dividing the speed bound by $d$.

\subsection{Forward growth}

Retain the profiles $Y_t$ and the quantity $S_t$ from
\eqref{eq:history}: $Y_t(r)=H_t-H_{t-r}$ for $r\le t$,
$Y_t(r)=\infty$ for $r>t$, and
$S_t=(\lambda/2)\Pp(W_r\le Y_t(r),\ r\ge0)$,
where $W$ is an independent rate-one symmetric walk on $\Z$, started at $0$.

Let $\Pi=\{x\in\Z^d:x_1=0\}$ and recall that
$\zeta=\inf\{t\ge0:X_t\cdot e_1=-1\}$ for a rate-$d$ walk $X$.
Write $H_{\{0\}}=\inf\{t\ge0:X_t=0\}$ and
$G_d^+(x,y)=\Ee_x\int_0^\zeta\one_{\{X_t=y\}}\,dt$.
Recall that $a_d=G_d^+(0,0)/2$.

\begin{lemma}\label{lem:excursions}
For every $0<\rho<1/a_d$, the forward increment estimate
\eqref{eq:increment} holds on $\Z^d$ in the rate-$d$ clock.
If $\lambda>a_d$, the conclusions of Lemmas \ref{lem:boost},
\ref{lem:good}, \ref{lem:profile-density} and \ref{lem:repetition}
also hold, with constants depending on $d$ and $\lambda$.
\end{lemma}

\begin{proof}
For $v\in\Pi$, put $q(v)=\Pp_v(H_{\{0\}}<\zeta)$.
The strong Markov property gives
$G_d^+(v,0)=q(v)G_d^+(0,0)$.
Symmetry gives $G_d^+(v,0)=G_d^+(0,v)$.
Summing over $v\in\Pi$ and interchanging the sum, integral and
expectation by monotone convergence, we obtain
\begin{equation}\label{eq:excursion-mass}
 \sum_{v\in\Pi}q(v)
 =\frac{\Ee_0\left[\int_0^\zeta\one_{\{X_t\cdot e_1=0\}}\,dt\right]}
        {G_d^+(0,0)}
 =\frac2{G_d^+(0,0)}=\frac1{a_d}.
\end{equation}
The second equality holds because, before reaching $-1$, the first
coordinate makes a geometric number of visits to $0$ with mean $2$,
each of mean duration $1$.

Fix $\rho<1/a_d$. By monotone convergence, there are finite
$D\subset\Pi$ and $L<\infty$ such that
$Q:=\sum_{v\in D}q_L(v)>\rho$, where
$q_L(v)=\Pp_v(H_{\{0\}}\le L,\ H_{\{0\}}<\zeta)$.

Condition on $\Gcal_\sigma$ at a finite stopping time $\sigma<T$.
Choose the lexicographically first rightmost vertex $U$ of $K_\sigma$.
Since $H_\sigma=h_{e_1}(\sigma)+1$, $U$ lies on the vacant outer boundary.
Select the prescribed trajectories whose first coordinate has stayed
strictly above $H_r$ for every $r\le\sigma$, and consider their
first subsequent visits to the fixed plane $x_1=H_\sigma$.
The trajectory exploration in the proof of Lemma \ref{lem:density}
leaves the free prescribed walks Poisson, restricted to paths avoiding
$K_r$ for $0\le r\le\sigma$.
Since staying ahead of $H_r$ implies this avoidance, the selected
trajectories have conditional intensity
$\lambda\sum_x\one_{\{\omega(r)\cdot e_1>H_r,\ 0\le r\le\sigma\}}
\Pp_x(d\omega)$, where $\Pp_x$ is the rate-$d$ walk law.
Mapping each selected trajectory to the time and site of its first
visit to the fixed plane, and using the calculation preceding
Lemma \ref{lem:increment}, shows that the crossings into
$U+v$, $v\in D$, form independent Poisson processes with common rate
\begin{equation*}
 S^\sigma(t)=\frac\lambda2\Pp\bigl(
 W_r\le H_\sigma-H_{(t-r)\wedge\sigma},\ 0\le r\le t\bigr),
 \qquad t\ge\sigma.
\end{equation*}
This rate is deterministic given $\Gcal_\sigma$ and equals $S_t$
before the next increase of $H$ or explosion.

A crossing into $U+v$ is selected if its prescribed continuation
hits $U$ within time $L$ without entering $\{x:x_1<H_\sigma\}$.
By the strong Markov property at the crossing time, its selection
probability is $q_L(v)$.
The selected crossing processes are therefore independent Poisson
processes with rates $q_L(v)S^\sigma(t)$
\cite[Theorems~5.2 and 5.6]{LP}.
Each selected crossing forces a height increase within time $L$,
unless explosion occurs first. Indeed, absorption after $\sigma$
and before the crossing, or during the prescribed continuation,
already adds a site with first coordinate at least $H_\sigma$.
If no such absorption occurs, the particle reaches $U$, which is
still on the outer boundary.

Let $\gamma>0$ be the density threshold in \eqref{eq:increment}, and put
\begin{equation*}
 c_\sigma=\inf\{t\ge\sigma:S^\sigma(t)<\gamma\},
 \qquad
 \theta_\sigma=T\wedge c_\sigma\wedge
 \inf\{t\in[\sigma,T):H_t>H_\sigma\}.
\end{equation*}
If $S^\sigma(\sigma)<\gamma$, then $\theta_\sigma=\sigma$.
Otherwise, assign to each selected crossing an independent uniform
variable on $[0,1]$, independent of all trajectories.
Before $c_\sigma$, retain a crossing at time $t$ precisely when
its variable is at most $\gamma/S^\sigma(t)$.
This thinning makes the union of retained crossings a Poisson process
of rate $\sum_{v\in D}q_L(v)\gamma=Q\gamma$ on $[\sigma,c_\sigma)$.
Continue it after $c_\sigma$ with an independent Poisson process
of the same rate. Its first waiting time $E$ has law
$\operatorname{Exp}(Q\gamma)$.
Since $\theta_\sigma\le c_\sigma$, either stopping precedes this
first point or that point forces stopping within another $L$ units
of time. Thus $\theta_\sigma-\sigma\le E+L$ in this coupling.

Consequently, for $0<\vartheta<Q$,
\begin{equation*}
 \Ee[e^{\gamma\vartheta(\theta_\sigma-\sigma)}\mid\Gcal_\sigma]
 \le M_\gamma(\vartheta):=
 e^{\gamma\vartheta L}\frac{Q}{Q-\vartheta}.
\end{equation*}
Since $\rho<Q$, for sufficiently small $\vartheta>0$ there are
$c,\gamma_0>0$ such that
\begin{equation*}
 \rho\log M_\gamma(\vartheta)\le\vartheta-c,
 \qquad\text{for all }0<\gamma\le\gamma_0.
\end{equation*}
The conditional iteration in Lemma \ref{lem:increment} now proves
\eqref{eq:increment}. The auxiliary selections are not included in
$\Gcal$; each comparison is conditioned on the original aggregate
history, and no independence of successive waiting times is asserted.

If $\lambda>a_d$, choose $1/\lambda<\rho<1/a_d$, then
$\varepsilon>0$ small enough that $\rho\lambda(1-\varepsilon)>1$,
and finally $1<\kappa<\rho\lambda(1-\varepsilon)$.
Lemma \ref{lem:barrier} concerns only the rate-one first coordinate.
Thus the proofs of Lemmas \ref{lem:boost} and \ref{lem:good} apply
with these constants, and Lemma \ref{lem:profile-density} gives the
same density bound ahead of $H$.
Finally, the proof of Lemma \ref{lem:repetition} uses its assumption
$\lambda>2/3$ only to invoke Lemma \ref{lem:good}; it therefore applies
under $\lambda>a_d$ as well.
\end{proof}

For $a_d<\lambda<\lambda_c(d)$, the nonexplosion argument in
Subsection \ref{subsec:dimensions-subcritical} gives global existence.
Together with Lemma \ref{lem:excursions}, this supplies the inputs
to the recovery argument in the proof of Theorem \ref{thm:linear}.
It gives positive lower speed in direction $e_1$.
Coordinate permutations and reflections give the result simultaneously
in all $2d$ coordinate directions. Dividing the speed by $d$ proves
part (2) of Corollary \ref{cor:dimensions}.

\subsection{Comparison with the percolation threshold}

\begin{lemma}\label{lem:ad-all-dimensions}
For every $d\ge2$, $a_d<\lambda_c(d)$.
\end{lemma}

\begin{proof}
For $d=2$, the probabilities in \eqref{eq:excursion-mass} satisfy
$q(0)=1$, $q(\pm e_2)\ge1/4$ and $q(2e_2)>0$.
Hence $1/a_2>3/2$, and Lemma \ref{lem:percolation} gives
$a_2<2/3<\log2\le\lambda_c(2)$.

For $d\ge3$, the path-counting argument in
\cite[proof of Theorem~(1.10) and Eq.~(1.13)]{Grimmett}
gives $p_c(d)\ge1/(2d-1)$, and hence
\begin{equation}\label{eq:lambda-d-lower}
 \lambda_c(d)\ge\ell_d:=\log\frac{2d-1}{2d-2}
 >\frac1{2d}+\frac3{8d^2}+\frac7{24d^3}.
\end{equation}
The strict inequality follows by expanding
$\log(1-x)-\log(1-2x)$ at $x=1/(2d)$; all its coefficients are
positive. We bound $a_d$ below $\ell_d$, first for $d\ge15$ and
then for the remaining dimensions.

Let $I_k$ denote the modified Bessel function of the first kind, and
put $\phi(t)=e^{-t}I_0(t)$ and $\psi(t)=e^{-t}I_1(t)/t$, with
$\psi(0)=1/2$.
Each coordinate of $X$ is a rate-one walk, with transition
probabilities $\Pp_0(W_t=k)=e^{-t}I_k(t)$.
Reflection at $-1$, independence of the coordinates, and the identity
$I_0(t)-I_2(t)=2I_1(t)/t$ therefore give
\begin{equation}\label{eq:ad-integral}
 a_d=\int_0^\infty f_d(t)\,dt,
 \qquad f_d(t)=\phi(t)^{d-1}\psi(t).
\end{equation}
The Bessel formulas used here and below are
\cite[Eqs.~10.25.2, 10.29.1 and 10.32.3]{DLMF}.
Their integral representations, with an integration by parts for $\psi$,
give
\begin{equation*}
 \phi(t)=\frac1\pi\int_0^\pi e^{-t(1-\cos\theta)}\,d\theta,
 \qquad
 \psi(t)=\frac1\pi\int_0^\pi
 e^{-t(1-\cos\theta)}\sin^2\theta\,d\theta.
\end{equation*}
Thus $\phi$, $\psi$ and $f_d$ are decreasing.

Suppose first that $d\ge15$.
The Bessel series give $I_0(t)\le e^{t^2/4}$ and
$I_1(t)/t\le I_0(t)/2$, so
$f_d(t)\le e^{-dt+dt^2/4}/2$.
Using $e^x\le1+x+x^2e^x/2$ for $x\ge0$ and $t^2\le t$ on $[0,1]$,
we obtain
\begin{equation*}
 \int_0^1 f_d(t)\,dt
 \le\frac1{2d}+\frac1{4d^2}
       +\frac{d^2}{64}\int_0^\infty t^4e^{-3dt/4}\,dt
 =\frac1{2d}+\frac1{4d^2}+\frac{128}{81d^3}.
\end{equation*}
Since $\phi(1)\le e^{-3/4}<1/2$ and $f_d=\phi^{d-2}f_2$,
the bound $a_2<2/3$ gives
$\int_1^\infty f_d(t)\,dt<(2/3)2^{2-d}$.
Consequently,
\begin{equation*}
 a_d<\frac1{2d}+\frac1{4d^2}+\frac{128}{81d^3}
          +\frac23\,2^{2-d}
 <\frac1{2d}+\frac3{8d^2}<\ell_d.
\end{equation*}
Indeed, the middle inequality is equivalent to
$128/(81d)+(2/3)d^22^{2-d}<1/8$.
Both terms decrease for $d\ge15$, and at $d=15$ their sum is
$128/1215+75/4096<1/8$.

It remains to check $3\le d\le14$.
We give explicit rational upper bounds for \eqref{eq:ad-integral}.
The inequalities $1-\cos\theta\ge2\theta^2/\pi^2$ on $[0,\pi]$
and $I_1(t)\le I_0(t)$ imply
$\phi(t)\le\sqrt\pi/(2\sqrt{2t})$ and $\psi(t)\le\phi(t)/t$.
Hence
$\int_{20}^\infty f_d(t)\,dt
\le(2/d)(\pi/160)^{d/2}<2/(d7^d)$.

For $\nu=0,1$, define the polynomials
\begin{equation*}
 P_\nu(t)=\frac1{2^\nu}\left[
 \sum_{k=0}^{100}\frac{(t^2/4)^k}{k!(k+\nu)!}
 +\frac{(t^2/4)^{100}}{99\,100!(100+\nu)!}\right],
 \qquad E(t)=\sum_{k=0}^{100}\frac{t^k}{k!}.
\end{equation*}
For $0\le t\le20$, the ratio of each omitted Bessel term to its
predecessor is less than $1/100$.
The omitted sum is therefore at most the last term in the summation
divided by $99$. Thus $I_0(t)\le P_0(t)$, $I_1(t)/t\le P_1(t)$ and
$E(t)\le e^t$.
Since $f_d$ decreases, its left Riemann sum gives
\begin{equation}\label{eq:ad-certificate}
 a_d\le U_d:=
 \frac1{1000}\sum_{j=0}^{19999}
 \frac{P_0(j/1000)^{d-1}P_1(j/1000)}{E(j/1000)^d}
 +\frac2{d7^d}.
\end{equation}
All terms in $U_d$ are rational. Rounding the numerator polynomials
upwards and the denominator polynomial downwards gives the following
rational upper bounds:
\begin{center}
\begin{tabular}{c|c@{\qquad}c|c}
$d$ & Upper bound for $U_d$ & $d$ & Upper bound for $U_d$\\ \hline
3 & $0.211564$ & 9  & $0.059309$\\
4 & $0.147115$ & 10 & $0.053047$\\
5 & $0.113153$ & 11 & $0.047991$\\
6 & $0.092098$ & 12 & $0.043820$\\
7 & $0.077723$ & 13 & $0.040321$\\
8 & $0.067266$ & 14 & $0.037343$
\end{tabular}
\end{center}
Each entry is strictly smaller than
$1/(2d)+3/(8d^2)+7/(24d^3)$.
Equation \eqref{eq:lambda-d-lower} completes the proof.
\end{proof}

\subsection{Supercritical explosion}

Fix $\lambda>\lambda_c(d)$. Lemmas \ref{lem:ad-all-dimensions} and
\ref{lem:excursions} give the good-history and density estimates
needed in the proof of Theorem \ref{thm:supercritical}.
We now construct the trial for explosion in $\Z^d$.

For $p>p_c(d)$, let $\theta_{H,d}(p)$ be the probability that a fixed
boundary vertex belongs to an infinite open cluster within a coordinate
half-space of $\Z^d$. It is positive for $p>p_c(d)$:
Grimmett--Marstrand \cite[p.~440]{GM} prove equality of the half-space
and full-space thresholds for $d\ge3$, and
Lemma \ref{lem:percolation} gives the planar case.

Fix deterministic $s$ and suppose $s<T$ and $Y_s\ge z_{\alpha_*}$.
Choose the lexicographically first rightmost vertex $U$ of $K_s$
and put $x_j=U+je_1$ for $j\ge1$.
Write $q_j=1-b(j)$, with $b$ from Lemma \ref{lem:profile-density}.
Then $q_1>0$, $q_j\uparrow1$, and
$u_s(x)\ge\lambda q_j$ for $x_1\ge H_s+j$.
Choose $M\ge1$ and then $h>0$ so that
$p_M:=1-\exp\{-\lambda q_Me^{-dh}\}>p_c(d)$.
Here $e^{-dh}$ is the probability that a rate-$d$ particle makes
no jump during an interval of length $h$. Thus $p_M$ is a lower
bound on the probability that a site at distance at least $M$
ahead of $H_s$ contains such a stationary particle.

Among particles free at $s$, let $N^0(x)$ count those at $x$ whose
prescribed walks make no jump in $(s,s+h]$, and let $N^-(x)$ count
those whose first jump occurs in that interval in direction $-e_1$.
The two conditions on a particle's future walk are disjoint.
Poisson thinning therefore gives, conditionally on $\Gcal_s$,
independence among all counts $N^0(x),N^-(x)$, with respective means
$u_s(x)e^{-dh}$ and $u_s(x)(1-e^{-dh})/(2d)$.
Put $O^0=\{x:N^0(x)\ge1\}$ and $D_M=\{x:x_1\ge H_s+M\}$.
Consider the events
\begin{equation*}
 \begin{aligned}
 F_\infty&=\{|C(x_M,O^0\cap D_M)|=\infty\},\\
 F_{\rm seg}&=\bigcap_{j=1}^{M-1}\{N^0(x_j)\ge1\},
 \qquad F_{\rm jump}=\{N^-(x_1)\ge1\}.
 \end{aligned}
\end{equation*}
These events are conditionally independent, since they use disjoint
sets of the conditional Poisson particles. Put
$F=F_\infty\cap F_{\rm seg}\cap F_{\rm jump}$.
The occupied field $O^0\cap D_M$ dominates Bernoulli site percolation
of parameter $p_M$ on $D_M$, so $\Pp(F\mid\Gcal_s)\ge p$ with
\begin{equation*}
 p:=\theta_{H,d}(p_M)
 \left[1-e^{-\lambda q_1(1-e^{-dh})/(2d)}\right]
 \prod_{j=1}^{M-1}\left[1-e^{-\lambda q_je^{-dh}}\right]>0,
\end{equation*}
where an empty product equals one.

On $F$, the particles counted by $N^0$ connect $x_1$ to an infinite
occupied set and make no jump before $s+h$.
The first absorption from this connected set is an infinite engulfment.
If $T>s+h$, these particles must therefore remain free throughout
$(s,s+h]$. The particle specified by $F_{\rm jump}$ then forces
contact by its jump from $x_1$ into $U$. Absorbing that particle
or $U$ earlier would already force contact with the stationary set.
This contradiction proves $F\subseteq\{T\le s+h\}$ and hence
$\Pp(T\le s+h\mid\Gcal_s)\ge p$ on the good-history event.

To repeat the trial, put $r=s+L(s)$, with $L(s)$ given by
Lemma \ref{lem:good}, as extended to $\lambda>a_d$ in
Lemma \ref{lem:excursions}. On $\{s<T\}$, the tower property gives
\begin{align*}
 \Pp(T\le r+h\mid\Gcal_s)
 &=\Ee[\Pp(T\le r+h\mid\Gcal_r)\mid\Gcal_s]\\
 &\ge\Pp(T\le r\mid\Gcal_s)
   +p\Pp(r<T,\ Y_r\ge z_{\alpha_*}\mid\Gcal_s)
 \ge p/2.
\end{align*}
As in Lemma \ref{lem:repetition}, set $s_0=0$ and
$s_{n+1}=s_n+L(s_n)+h$.
These times are deterministic and tend to infinity because
$s_{n+1}-s_n\ge h$. Iterated conditioning gives
$\Pp(T>s_n)\le(1-p/2)^n$, so $T<\infty$ almost surely.

\subsection{Infinite engulfment at supercriticality}

We now show that, for $\lambda>\lambda_c(d)$, explosion occurs by
engulfing an infinite cluster. A component that reaches the current
supporting hyperplane has a positive conditional probability of
connecting to infinity through the unexamined field ahead.
On a bounded time interval this probability is uniform over the
history. Accumulation would require infinitely many advances in
some coordinate direction; the conditional estimate below excludes
this without assuming independence between successive attempts.

\begin{lemma}\label{lem:no-accumulation-d}
For $d\ge2$ and $\lambda>\lambda_c(d)$, $\Pp(T<T_I)=0$.
\end{lemma}

\begin{proof}
Fix $0<t_0<\infty$. For $s<T$ with $s\le t_0$,
$K_r\subseteq\{x:x_1\le H_s\}$ for every $r\le s$.
The first coordinate of the reversed walk has rate one, so
Lemma \ref{lem:density} gives, for every integer $L\ge1$,
\begin{equation}\label{eq:bounded-time-forward-density}
 \begin{aligned}
 u_s(x)&\ge\lambda e^{-t_0},
 &&x_1>H_s,\\
 u_s(x)&\ge\lambda[1-\delta_L],
 &&x_1\ge H_s+L,
 \qquad \delta_L:=\Pp(\operatorname{Pois}(t_0)\ge L).
 \end{aligned}
\end{equation}
Indeed, zero horizontal jumps suffice for the first bound, and
fewer than $L$ suffice for the second.
Only $K_r\subseteq\{x:x_1\le H_s\}$ was used, so no good-history
assumption is needed. At an entrance $t\le t_0$, the same argument
with the pre-update history gives these bounds for $u_{t-}$ and
$H_{t-}$: all earlier queried sets lie in
$\{x:x_1\le H_{t-}\}$.

Choose $L$ such that $p_*:=1-e^{-\lambda(1-\delta_L)}>p_c(d)$,
and put
$p_0=1-e^{-\lambda e^{-t_0}}$,
$\varepsilon=p_0^{L-1}\theta_{H,d}(p_*)$ and $q=1-\varepsilon$.
The half-space percolation estimate above gives $0<\varepsilon<1$.
These constants depend only on $t_0,\lambda,d$.

Consider an entrance at $t\le t_0$ with finite pre-update aggregate.
Condition on the preceding exploration, $t$, and the directed
entrance edge, and denote the corresponding $\sigma$-field by $\mathcal H$.
Put $H=H_{t-}$ and let $w$ be the arrival site.
The entrance law following Lemma \ref{lem:density} represents the
occupied set immediately after the jump, before the component search, as
$O=\{x:\xi(x)\ge1\}\cup\{w\}$, where, conditionally on $\mathcal H$,
the $\xi(x)$ are independent Poisson variables of means $u_{t-}(x)$.
Since $w\in\partial^o\Ecal_{t-}$, we have $w_1\le H$. Set
\begin{equation*}
 C=C(w,O),\qquad
 C^-=C(w,O\cap\{x:x_1\le H\}),\qquad
 B=\{C^-\cap\{x:x_1=H\}\ne\varnothing\}.
\end{equation*}
Let $\mathcal H^-=\mathcal H\vee\sigma(\xi(x):x_1\le H)$.
On $B$, select $v\in C^-\cap\{x:x_1=H\}$ using a fixed enumeration
of $\Z^d$. Conditional on $\mathcal H^-$, the counts ahead of the
plane remain independent Poisson variables with means $u_{t-}(x)$.
Put $D_L=\{x:x_1\ge H+L\}$ and define
\begin{equation*}
 F_v=\{v+je_1\in O\text{ for }1\le j<L\}
       \cap\{|C(v+Le_1,O\cap D_L)|=\infty\}.
\end{equation*}
The two requirements use disjoint sets of sites. By
\eqref{eq:bounded-time-forward-density}, on $B$,
$\Pp(F_v\mid\mathcal H^-)\ge p_0^{L-1}\theta_{H,d}(p_*)
=\varepsilon$. Moreover, $B\cap F_v\subseteq\{|C|=\infty\}$.

Let $I=\{|C|=\infty\}$ and let $J$ be the event of a finite update
increasing $H$.
Since the old aggregate has maximum first coordinate $H-1$,
\begin{equation*}
 J=\{|C|<\infty,\ \max_{x\in C}x_1\ge H\}=I^c\cap B.
\end{equation*}
The last equality follows by stopping a path from $w$ at its first
visit to the plane $x_1=H$.
Conditioning from $\mathcal H^-$ back to $\mathcal H$ gives
\begin{equation}\label{eq:record-versus-infinite}
 \begin{aligned}
 \Pp(I\cap B\mid\mathcal H)&\ge\varepsilon\Pp(B\mid\mathcal H),\\
 \Pp(J\mid\mathcal H)&\le q\Pp(B\mid\mathcal H)
 \le\frac q\varepsilon\Pp(I\mid\mathcal H).
 \end{aligned}
\end{equation}

If the initial engulfment is infinite, $T=T_I=0$.
Otherwise, condition on a finite initial exploration.
Stop the update construction at an infinite engulfment or when
the next entrance is later than $t_0$.
Let $\mathcal F_n$ record its first $n$ steps, with the construction
kept fixed after stopping.
The auxiliary half-space revelation is not included in $\mathcal F_n$.
Among these $n$ steps, let $N_n$ count finite updates increasing $H$,
and let $A_n$ be the event that no infinite engulfment has occurred.
At a step with an entrance by $t_0$, \eqref{eq:record-versus-infinite}
gives
\begin{equation*}
 \Ee[\one_{I^c}q^{-\one_J}\mid\mathcal H]
 =1-\Pp(I\mid\mathcal H)
   +\frac{\varepsilon}{q}\Pp(J\mid\mathcal H)\le1.
\end{equation*}
Define $Z_n=q^{-N_n}\one_{A_n}$. Conditioning also over the next
entrance time and edge gives
$\Ee[Z_{n+1}\mid\mathcal F_n]\le Z_n$, with $Z_0=1$.
After stopping, this inequality is an equality. Thus $(Z_n)$ is a
nonnegative supermartingale.

Put $\rho_m=\inf\{n:N_n\ge m\}$.
On $\{\rho_m<\infty\}$, $Z_{\rho_m}=q^{-m}$.
Optional stopping at $\rho_m\wedge n$ gives
\begin{equation*}
 q^{-m}\Pp(\rho_m\le n)
 \le\Ee Z_{\rho_m\wedge n}\le1,
 \qquad
 \Pp(\rho_m<\infty)\le q^m.
\end{equation*}
The same bound holds in every signed coordinate direction.

On $\{T\le t_0,\ T<T_I\}$, infinitely many finite updates occur
by $t_0$, each adding a vertex.
Their union is therefore unbounded, so some signed coordinate
maximum makes infinitely many advances.
Consequently, for every $m\ge1$,
\begin{equation*}
 \Pp(T\le t_0,\ T<T_I)\le2d\,q^m.
\end{equation*}
Letting $m\to\infty$ gives
$\Pp(T\le t_0,\ T<T_I)=0$ for each fixed $t_0<\infty$.
Taking the union over positive integer $t_0$ proves
$\Pp(T<T_I)=0$. All density estimates were applied at finite
updates, strictly before any possible accumulation time.
\end{proof}

Together with the preceding proof that $T<\infty$, this gives
$T=T_I<\infty$ almost surely for every $\lambda>\lambda_c(d)$.

\section*{Acknowledgments}
We warmly thank Yaosen Liu for discussions in the model setting. YZ is supported by NSFC-12271010, National Key R\&D Program of China (No.\ 2020YFA0712902) and the Fundamental Research Funds for the Central Universities and the Research Funds of Renmin University of China 24XNKJ06.

\bibliography{bibliography}
\bibliographystyle{plain}
\end{document}